\documentclass{amsart}
\usepackage{amsmath, amssymb, amscd, amsthm, amsfonts, amstext,
  amsbsy, enumerate, mathrsfs}

\newcommand{\seqo}{{}^{<\omega}\omega}
\newcommand{\fhi}{\varphi}
\newcommand{\imp}{\Rightarrow}
\renewcommand{\L}{\mathcal{L}}

\newtheorem{theorem}{Theorem}[section]
\newtheorem{lemma}[theorem]{Lemma}
\newtheorem{corollary}[theorem]{Corollary}
\newtheorem{proposition}[theorem]{Proposition}

\newtheorem{question}{Question}
\newtheorem*{Vaughtconj}{Vaught's Conjecture}
\theoremstyle{definition}
\newtheorem{claim}{Claim}[theorem]
\newtheorem{defin}{Definition}

\newtheorem{example}{Example}
\theoremstyle{remark}
\newtheorem{remark}[theorem]{Remark}

\begin{document}

\title{Analytic equivalence relations and bi-embeddability}
\date{\today}
\author{Sy-David Friedman and Luca Motto Ros}
\date{\today}
\address{Kurt G\"odel Research Center for Mathematical Logic\\
  University of Vienna \\ W\"ahringer Stra{\ss}e 25 \\ 
 A-1090 Vienna \\
Austria}
\email{sdf@logic.univie.ac.at \\ luca.mottoros@libero.it}
\thanks{The authors would like to thank the FWF (Austrian Research
  Fund) for generously supporting this research through  Project
  number P 19898-N18.} 

\maketitle

\begin{abstract}
 Louveau and Rosendal \cite{louros} have shown that the relation of
bi-embeddability for countable graphs as well as for many
other natural classes of countable structures is complete
under Borel reducibility for analytic equivalence relations.
This is in strong contrast to the case of the isomorphism
relation, which as an equivalence relation on graphs (or on
any class of countable structures consisting of the models
of a sentence of $\L_{\omega_1\omega}$) is far from complete (see
\cite{louros, friedmanstanley}). 

In this article we strengthen the results of \cite{louros} by showing that
not only does bi-embeddability give rise to analytic equivalence
relations which are complete under Borel reducibility, but in fact
\emph{any} analytic equivalence relation is \emph{Borel
equivalent} to such a relation. This result and the techniques introduced 
 answer questions raised in \cite{louros} 
about the comparison between isomorphism and bi-embeddability.
Finally, as in \cite{louros} our results apply
not only to classes of countable structures defined by
sentences of $\L_{\omega_1\omega}$, but also to discrete
metric or ultrametric Polish spaces, compact metrizable
topological spaces and separable Banach spaces, with various
notions of embeddability appropriate for these classes, as well as to
actions of Polish monoids.
\end{abstract}

\section{Introduction}

Given two $n$-ary analytic (i.e.\ $\boldsymbol{\Sigma}^1_1$) relations
$R , R'$ defined on standard Borel spaces $X,X'$, we say that $R$ is
\emph{Borel reducible} to $R'$ ($R \leq_B R'$ in symbols) if there is
a Borel measurable function 
$f \colon X \to X'$ such that
\[ \forall x_1, \dotsc, x_n \in X (R(x_1, \dotsc , x_n) \iff R'(f(x_1),
\dotsc , f(x_n)).\]
We say that $R$ and $R'$ are \emph{Borel equivalent} ($R \sim_B R'$ in
symbols) if both $R \leq_B
R'$ and $R' \leq_B R$, and that $R$ is
\emph{complete} if $R' \leq_B R$ for any $n$-ary analytic
relation $R'$.

The quasi-order $\leq_B$ can be used to measure the relative
complexity of the relations $R$ and $R'$, and has been used
(explicitly or implicitly) in many areas of mathematics to solve
various classification problems --- see \cite{hjorth} for
a brief but informative discussion on this topic. For this reason, the
relation $\leq_B$ has been extensively studied in the case $n=1$, and
in the case $n=2$ when  restricting either to quasi-orders (reflexive and transitive relations) or to
equivalence relations (symmetric quasi-orders).
Logic provides a nice example of analytic equivalence relation, namely the isomorphism
relations (denoted by $\cong$) on the class of countable models of some
$\L_{\omega_1 
  \omega}$-sentence, $\L$ a countable language: unfortunately, it turns
out that these relations are in a sense too special, as there are other relations (such as $E_1$) which are 
not even Borel reducible to $\cong$. This means the the relation of
isomorphism is very far from being able to produce (up to Borel equivalence) all analytic equivalence
relations.
More generally, one can see that the latter
gives rise to an extremely complex structure under Borel reducibility, in which many
``pathologies'' (such as e.g.\ infinite descending chains, or infinite
antichains) appear, even in very simple subclasses.

Regarding completeness, there are of course
complete analytic equivalence relations, but, to our
knowledge, the unique technique to produce examples which are
``concrete'' from the mathematical point of view is to first find
examples of 
complete analytic quasi-orders $R$ and then pass to the associated
equivalence relation $E_R  = R \cap R^{-1}$, a method developed by
Louveau and Rosendal in their \cite{louros}. Among other things, in
that paper they
proved quite surprisingly that when considering the natural
counterpart of isomorphism, namely embeddability, on countable graphs,
one gets a complete analytic quasi-order (denoted by $\sqsubseteq$):
this in particular implies that the analytic equivalence relation
of bi-embeddability (denoted\footnote{We would like to warn the reader that in this paper, as in \cite{louros}, the symbol $\equiv$ denotes bi-embeddability and \emph{not} elementary equivalence.} by $\equiv$) on countable graphs is complete. Building
on the work of Louveau and Rosendal, we show in this paper that the
embeddability relation has 
 the further stronger ``universal'' property that \emph{any} analytic
 quasi-order 
is indeed Borel equivalent to $\sqsubseteq$ restricted to the class of countable graphs
satisfying a corresponding $\L_{\omega_1 \omega}$-sentence, whence the
similar result regarding analytic equivalence relations and
bi-embeddability. As we will see in Remark \ref{remeffective}, there is an \emph{effective} version of this result. Moreover, following \cite{louros}, the same result can be
 extended to the context of
analysis and monoid actions.

We finally want to stress that the Borel equivalences obtained here
are of a very special kind, and give rise to a stronger relationship
(which was already considered as an interestening notion in
\cite{friedmanstanley})  between
arbitrary analytic equivalence relations and bi-embeddability. In
fact, the
existence of a Borel reduction $f$ between
analytic equivalence relations $E$ and $F$ on standard Borel spaces
$X$ and $Y$ implies that its lifting $\hat{f} \colon X/_E \to Y/_F \colon
[x]_E \mapsto [f(x)]_F$ is well-defined and injective: therefore $E
\leq_B F$ if and only if there is a Borel $f \colon X \to Y$ such that
$\hat{f}$ is a well-defined embedding of $X/_E$ into $Y/_F$, and,
consequently,  $E
\sim_B F$ if and only if there are Borel functions $f \colon X \to Y$
and $g \colon Y \to X$ such that $\hat{f}$ and $\hat{g}$ are
well-defined embeddings of $X/_E$ into $Y/_F$ and of $Y/_F$ into
$X/_E$, respectively. Note, however, that with this definition the
embeddings $\hat{f}$ and $\hat{g}$ may have very little in common. A
stronger requirement (already proposed in \cite{friedmanstanley}) 
 would be to ask for reductions $f$ and $g$ such
that $\hat{g} = \hat{f}^{-1}$, so that $\hat{f}$ is indeed an
isomorphism between $X/_E$ and $Y/_F$: in this case $E$ and $F$ are
said to be \emph{Borel isomorphic}\footnote{Borel isomorphism is \emph{strictly} finer than Borel equivalence. In fact, consider a closed set $C \subseteq {}^\omega \omega \times {}^\omega \omega$, where ${}^\omega \omega$ denotes the Baire space, whose first projection is all of ${}^\omega \omega$ but $C$ has no Borel uniformization (see e.g.\ \cite[Exercise 18.17]{kechris}). Now define $E$ as the equivalence relation on the disjoint union of $C$ and any uncountable Polish space $X$ defined by $x E y$ if and only if either $x,y \in X$ and $x=y$, or else $x,y \in C$ and $p_0(x) = p_0(y)$, where $p_0$ denotes the projection on the first coordinate. It is clear that $E$ is Borel equivalent to the equality relation on $X$ (use $p_0$ to build a reduction of $E$ into equality on $X$, and the identity function for the other direction), but that they cannot be Borel isomorphic because a witness to this fact would also give a Borel uniformization of $C$.}.
As it can be easily checked, our construction will give that for every
analytic equivalence 
relation $E$ on a standard Borel space $X$ there is an $\L_{\omega_1
  \omega}$-sentence $\fhi$ such that $E$ is not only Borel equivalent,
but Borel isomorphic to the bi-embeddability relation on the class
$Mod_\fhi$ of countable models of $\fhi$ (and the witness of this fact
will also 
witness at the same time that equality on $X$ is Borel isomorphic to
isomorphism on 
$Mod_\fhi$). The same strong property will be provided also for the
cases regarding analysis and monoid actions.

\section{Notation and preliminaries}\label{sectionpreliminaries}

Let $\omega$ denote the set of natural numbers.
Given a nonempty set $A$, we will denote by ${}^{<\omega} A$ the collection of
all finite sequences $s$ of elements of $A$, and denote by $|s|$ its
length. For each $i<|s|$, $s(i)$ will denote the $(i+1)$-st
element of $s$, so that $s = \langle s(i) \mid i <
|s|\rangle$. Moreover, we will write $s {}^\smallfrown t$ for the
concatenation of $s$ and $t$ (but for simplicity of notation we will
simply write e.g.\ $n
{}^\smallfrown t$ rather than $\langle n \rangle {}^\smallfrown t$).
If $A = A_0 \times \dotsc \times A_k$, we will identify each element
$s \in {}^{< \omega}A$ with a sequence $(s_0, \dotsc, s_k)$ of
sequences of same length such that $s_i  \in {}^{<\omega} A_i$
for each $i \leq k$. A \emph{tree} $T$ on $A$ is simply a subset of ${}^{<
  \omega} A$ closed under initial subsequences. If $T$ is a tree on
$A$ we call the \emph{body} of $T$ the set $[T]$ of all $\omega$-sequences
$x$ of elements of $A$ such that $x \restriction n \in T$ for each $n
\in \omega$ (where $x \restriction n$ denotes the restriction of $x$
to its first $n$ digits). Moreover, if $T$ is a tree on $A \times
\omega$, the \emph{projection} $p[T]$ of $T$ is the collection of
$\omega$-sequences $x$ of elements of $A$ such that for some $\omega$-sequence $y$
of natural numbers $(x \restriction n, y \restriction n) \in T$ for
each $n \in \omega$. 

A special role in the paper will be played by the Cantor space
${}^\omega 2$ of all binary sequences of length
$\omega$: in fact, since any two uncountable Polish spaces are
Borel-isomorphic and all the notions considered here depend only on the
Borel-structure of the spaces involved, we can restrict our attention
without loss of 
generality to quasi-orders and
equivalence relations defined on ${}^\omega 2$.

Finally, we need to briefly recall some notation, definitions, and basic
results from \cite{louros}. Given two sequences $s,t \in \seqo$ of the
same length, we
put $s \leq t$ if $s(i) \leq t(i)$ for every $i < |s|=|t|$, and 
define $s+t \in \seqo$ by setting $(s+t)(i) = s(i)+t(i)$. Given a tree
$T$ on $X \times \omega$, we say that $T$ is \emph{normal} if
given $u \in {}^{<\omega}X$ and $s \in \seqo$, $(u,s) \in T$ implies
$(u,t) \in T$ for every $s \leq t$. It is well-known that any analytic
subset of ${}^\omega 2 \times {}^\omega 2$ (hence, in particular, any
analytic quasi-order $R$ on ${}^\omega 2$) is the projection of a tree
on $2 \times 2 \times \omega$: in  \cite[Theorem 2.4]{louros},
Louveau and
Rosendal proved the following stronger result, in which $R$ is viewed as the
projection of a normal tree $S$ on $2 \times 2 \times \omega$ such
that the reflexivity and transitivity properties of $R$ are
``mirrored'' by corresponding local properties of $S$.

\begin{proposition}[Louveau-Rosendal]\label{propnormalform}
  Let $R$ be any analytic quasi-order on ${}^\omega 2$. Then there is a
  normal tree $S$ on $2 \times 2 \times \omega$ such that:
  \begin{enumerate}[i)]
  \item $R = p[S]$;
\item for every $u \in {}^{< \omega}2$ and $s \in \seqo$ of the same
  length, $(u,u,s) \in S$;
\item for every $u,v,w \in {}^{<\omega}2$ and $s,t \in \seqo$ of the
  same length, if $(u,v,s) \in S$ and $(v,w,t) \in S$ then $(u,w,s+t)
  \in S$.
  \end{enumerate}
\end{proposition}

From this one can easily get the following refinement, which will be
needed for our result.

\begin{corollary}\label{cornormalform}
  Let $R$ be any analytic quasi-order on ${}^\omega 2$. Then there is a
  normal tree $S$ on $2 \times 2 \times \omega$ which satisfies
  i)--iii) of Proposition \ref{propnormalform} and
\begin{enumerate}[iv)]
 \item for every $u,v \in {}^{<\omega}2$ of the same
   length, $(u,v,0^{|u|}) \in S$ implies $u=v$.
\end{enumerate}
\end{corollary}

\begin{proof}
For $u \neq v \in {}^{<\omega}2$ of the same length, discard each
element of the form
$(u,v,0^{|s|})$ from the tree $S$ given by Proposition
\ref{propnormalform}, and  check that the new tree is normal and still
satisfies i)--iii).
\end{proof}

A function $f \colon \seqo \to \seqo$ is said to be \emph{Lipschitz}
if $s \subseteq t \imp f(s) \subseteq f(t)$ and $|s| = |f(s)|$ for
each $s,t \in \seqo$. Consider now the space $\mathcal{T}$ of all
normal trees on $2 \times \omega$, and for $S,T \in \mathcal{T}$ put
$S \leq_{max} T$ if and only if there exists a Lipschitz function $f
\colon \seqo \to \seqo$ such that $(u,s) \in S \imp (u,f(s)) \in T$
for every $u \in {}^{<\omega}2$ and $s \in \seqo$ of the same length
(this in particular implies $p[S] \subseteq p[T]$).

\begin{theorem}[Louveau-Rosendal]\label{theomax}
The quasi-order $\leq_{max}$ is complete for analytic quasi-orders.
\end{theorem}

In their proof, Louveau and Rosendal considered an arbitrary analytic
quasi-order $R$ on ${}^\omega 2$ together with the tree $S$ given by
Proposition \ref{propnormalform}: then they defined for every $x \in
{}^\omega 2$ the normal tree $S^x = \{(u,s) \mid (u,x \restriction
|u|,s) \in S \}$ 
and showed that $xRy$ if and only if $S^x \leq_{max} S^y$ (for $x,y
\in {}^\omega 2$). If we now use Corollary \ref{cornormalform} rather than
Proposition \ref{propnormalform} to carry out the same proof, we get
the additional property that
 the continuous map which sends $x$ to $S^x$ is now
injective (and moreover becomes also open in its range):
if $x \neq y$ and  $s$ is such
that $s \subseteq x$ but $s 
\nsubseteq y$ then $(s,0^{|s|}) \in S^x$ but $(s,0^{|s|}) \notin
S^y$. Therefore, given a quasi-order $R$ on ${}^\omega 2$ and $x \in
{}^\omega 2$, from 
now on we will denote by $S^x$ the normal tree defined as above
\emph{but using the tree $S$ given by Corollary
  \ref{cornormalform}}.\\

The next ingredient for our proof
is the collection $OCT$ of the \emph{ordered (countable)
  combinatorial trees}, i.e.\ the collection of those $G = \langle
U_G, G, \leq_G \rangle$ such that $\langle U_G, G \rangle$ is a
\emph{combinatorial tree} (that is a connected and  acyclic graph)
and $\leq_G$ is any partial transitive relation on $U_G$ which will be
called the \emph{order}\footnote{This evident abuse
of terminology is justified by the fact that in all constructions
below, $\leq_G$ will always be either a linear well-founded  order (in
the usual sense) or
its strict part.} of $G$. Note that $OCT$ ``includes''
the collection $CT$ of all combinatorial trees considered in
\cite{louros}, as each of them can
be canonically identified with the corresponding ordered combinatorial
tree obtained by adjoining the empty set as order (and this
identification clearly preserves the relations of embeddability and
isomorphism). By this identification, one immediately gets from
\cite[Theorem 3.1]{louros} that
$\sqsubseteq_{OCT}$, the relation of
embeddability on $OCT$, is  a
complete analytic quasi-order. However, we will need a slightly
different proof which will be given in Theorem
\ref{theorcomplete}. From the construction given
  in our new argument  one could also restrict attention to
  \emph{rooted} $OCT$ (ordered combinatorial trees with a special
  distinguished element called root), or put many restrictions on
  $\leq_G$,
requiring it to be reflexive and antisymmetric (i.e.\ an order in the
common sense),  connected (for each pair of elements $t,t' \in U_G$ there
is $t'' \in U_G$ such that $t'' \leq_G t,t'$),
 linear (for each pair $t,t' \in U_G$ either $t \leq_G t'$ or $t'
 \leq_G t$), well-founded (there is no $\leq_G$-descending chain) and
 so on. One obviously could also enlarge the class of
 structures under consideration and get that  the embeddability relation on any
 class of structures between $OCT$ and the collection of all $\L$-structures,
 where $\L = \{P,Q\}$ is a language with just two binary relations, is a complete
 analytic quasi-order.

\begin{theorem}\label{theorcomplete}
The relation $\sqsubseteq_{OCT}$ of embeddability on $OCT$ is
complete for analytic quasi-orders.
\end{theorem}

\begin{proof}
The proof is almost identical to the one of
 \cite[Theorem 3.1]{louros}. To each normal tree $T \in \mathcal{T}$ on $2
\times \omega$
associate a combinatorial tree $G_T$ defined as in \cite{louros}:
\begin{enumerate}[i)]
\item fix an enumeration $\theta \colon {}^{<\omega}2 \to \omega$ such
  that $|s| \leq |t|$ implies $\theta(s) \leq \theta(t)$; 
\item ``double'' the set ${}^{<\omega} \omega$, that is adjoin a new
  vertex $s^*$ for each $s \in {}^{<\omega}\omega \setminus \{
  \emptyset \}$,  and put an edge between $s^*$ and $s$, and between
  $s^*$ and the predecessor $s^-$ of 
$s$ as a sequence (this combinatorial tree, which does not depend on $T$, is denoted
by $G_0$); 
\item for each pair $(u,s) \in T$ add vertices $(u,s,x)$, where $x =
  0^{(k)}$ or $x = 0^{(2 \theta(u)+2)} {}^\smallfrown 1 {}^\smallfrown
  0^{(k)}$, and then link $(u,s,\emptyset)$ to $s$ and $(u,s,x)$ to
  $(u,s,x^-)$ (where $x^-$ is again the predecessor of $x$ as a sequence). 
\end{enumerate}

Now define the order
$\leq_T = \leq_{G_T}$ on $U_{G_T}$ in the following way:
for $s,t \in \seqo$ put $s
\preceq t$ if and only if $|s| < |t|$ or $|s| = |t|$ and $s
\leq_{lex}t$ (the symbol $\prec$ will denote the strict part of
$\preceq$).  Now for $g,g' \in U_{G_T}$, $s,t \in \seqo$, and $u,v,x,y \in
{}^{<\omega}2$  put $g \leq_T g'$ in each of (and only) the following cases:
\begin{itemize}
\item $g=s$ and either $g'=t^*$ or $g'=(v,t,y)$, or $g=s^*$ and $g'=(v,t,y)$;
\item $g=s$, $g'=t$, and $s \preceq t$;
\item $g=s^*$, $g'=t^*$, and $s \preceq t$;
\item $g=(u,s,x)$, $g'=(v,t,y)$ and
\[ (s \prec t) \vee ({s=t}\wedge {u \prec v}) \vee ({s=t} \wedge {u=v}
\wedge {x \preceq y}).\] 
\end{itemize}
\noindent
(Note in particular that $\leq_T$ is a well-founded linear order of
length $ \leq \omega^4$ but one could also define a suitable linear order of
type $\omega$ as well, although this would make it considerably more
difficult to check that the functions defined below are really
embeddings, i.e.\ that they preserve the orders.)

Now we will show that the map $T \mapsto G_T$ is a reduction of
 $\leq_{max}$ to $\sqsubseteq_{OCT}$.
Assume first that $S,T$ are normal trees on $2 \times \omega$ such that
$S \leq_{max} T$: as observed in Lemma 2.8 of \cite{marconerosendal}, this can be witnessed by a Lipschitz
$\leq_{lex}$-preserving function $f
\colon \seqo \to \seqo$, that is by an $f$ such that $s \preceq s' \iff f(s) \preceq
f(s')$ for every $s,s' \in \seqo$ (in particular $f$ is injective).
 Now embed $G_S$ into $G_T$ sending $s$ to $f(s)$, $s^*$ to
$f(s)^*$, and $(u,s,x)$ to $(u,f(s),x)$, and check that both the
graph and the order relations  are preserved.

For the other direction, if $G_S
\sqsubseteq G_T$ than $G_S$ embeds in $G_T$ as a combinatorial tree
(disregarding the orders) and so $S \leq_{max} T$ by the second part
of the proof of
 \cite[Theorem 3.1]{louros}.
\end{proof}

 \begin{remark}\label{remarkoc}
Let us note here that if the coding of $G_T$ as a structure on
$\omega$ is chosen in a careful way, e.g.\ as in the proof of Theorem
\ref{theor2}, then the map which sends an arbitrary normal tree $T$
into (the code of) $G_T$, which is clearly Borel, has very low
topological complexity: in fact, it is continuous and open in its image.
\end{remark}

\section{The main results}\label{sectionmain}

We now want to prove our main results, namely that there are various
natural quasi-orders arising in model theory, analysis and descriptive
set theory such 
that each analytic
quasi-order  is
indeed Borel equivalent to that specific quasi-order (on a suitable
class of objects). This gives also several characterizations of both 
analytic quasi-orders and analytic equivalence relations, and shows that
the notions of embeddability, homomorphism, and weak-homomorphism
among countable structures (for model 
theory), the notions of isometric embeddability among discrete metric
or ultrametric Polish 
spaces, of continuous embeddability among
compact metrizable topological spaces, and of linear isometric
embeddability among separable Banach spaces (for analysis), and the
notion of closed or Borel action of  Polish monoids (for
descriptive set theory) are able to 
capture the 
great complexity of the whole structure of analytic quasi-orders and analytic
equivalence relations (up to  Borel equivalence).

\subsection{Morphisms in Model Theory}\label{sectionmodeltheory}

The advantage of having used $OCT$ in the previous section (rather
than $CT$ as in \cite{louros}) emanates from the following
two lemmas.

\begin{lemma}\label{lemma1}
 Let $S,T$ be normal trees, and $G_S$ and $G_T$ be defined as in the
 proof of Theorem \ref{theorcomplete}. If $S \neq T$ then $G_S \not\cong G_T$.
\end{lemma}

\begin{proof}
 Suppose $i$ is an isomorphism between $G_S$ and $G_T$. Since the
 orders $\leq_S$ and $\leq_T$ coincide on $\seqo$ we have that $i
 \restriction \seqo$ must be the identity. Suppose now $(u,s) \in S$:
 as in the proof of Theorem 3.1 in \cite{louros}, the point $(u,s,0^{2
   \theta(u)+2})$ must be sent to a point of the form $(u,i(s),0^{2\theta(u)+2}) =
 (u,s,0^{2 \theta(u)+2})$, and the existence of such a point witnesses $(u,s) \in T$. Hence $S
 \subseteq T$. Exchanging the role of $S$ and $T$ and using
 $i^{-1}$ instead of $i$ one gets $T \subseteq S$, and therefore $S=T$.
\end{proof}

As already noted in \cite{louros}, the domain of each ordered
combinatorial tree of the form $G_T$ is formally
different from  $\omega$,
but nevertheless one can easily code (Borel-in-$T$) such a structure in another
structure $\hat{G}_T$ with domain $\omega$: for simplicity of
presentation, as in the following lemma, we will often confuse the
two structures $G_T$ and $\hat{G}_T$.
Let $S_\infty$ be the Polish group of permutations on
$\omega$, $\mathcal{L} = \{P,Q\}$ be the relational language with just two
binary symbols,
and $j_\mathcal{L} \colon S_\infty \times Mod_{\mathcal{L}} \to
  Mod_{\mathcal{L}}$ be the usual (continuous) action of $S_\infty$ on
    $Mod_\mathcal{L}$, the collection of all countable
    $\L$-structures. For every normal tree $S$ on $2 \times
    \omega$ and $p \in S_\infty$, put $G_{S,p} =
    j_{\mathcal{L}}(p,G_S)$, where $G_S$ is the ordered combinatorial
    tree obtained from $S$ as in the proof of Theorem \ref{theorcomplete}.

\begin{lemma}\label{lemma2}
 For every distinct $p,q \in S_\infty$ and every  normal
 tree $S$ on $2 \times
 \omega$, we have $G_{S,p} \neq G_{S,q}$.
\end{lemma}

\begin{proof}
 Let $\leq_{S,p}$ and $\leq_{S,q}$ be the linear orders on $G_{S,p}$ and
 $G_{S,q}$, respectively. Let $g$ be the $\leq_S$-minimal element of
 $G_S$ such that $p(g) \neq q(g)$. We claim that $p(g) \leq_{S,p}
 q(g)$ but $p(g) \nleq_{S,q} q(g)$ (this implies that the two
 structures $G_{S,p}$ and $G_{S,q}$ are different). Assume toward a
 contradiction that $p^{-1}(q(g)) <_S g$: then $q(p^{-1}(q(g)))
 <_{S,q} q(g)$. But as $p(p^{-1}(q(g))) = q(g)$, the previous
 inequality shows that $p(p^{-1}(q(g))) \neq q (p^{-1}(q(g)))$,
 contradicting the $\leq_S$-minimality of $g$. Therefore $g \leq_S
 p^{-1}(q(g))$, which implies $p(g) \leq_{S,p} q(g)$.

Assume now towards a contradiction that $p(g) \leq_{S,q} q(g)$. Since
$p(g) \neq q(g)$ (by hypothesis) we get $p(g) <_{S,q} q(g)$, which
implies $q^{-1}(p(g)) <_S g$. Arguing as before (with $p$ and $q$
exchanged), we get a contradiction with the $\leq_S$-minimality of
$g$. Therefore $p(g) \nleq_{S,q} q(g)$, as required.
\end{proof}

Now we are ready to prove our first main theorem.

\begin{theorem}\label{theor1}
 If $R$ is an analytic quasi-order on a standard Borel space $X$, then
 there is an $\mathcal{L}_{\omega_1\omega}$-sentence $\fhi$  
   such that $R$ is
 Borel equivalent to embeddability on $Mod_{\fhi} = \{ x \in
 Mod_{\mathcal{L}} \mid x \vDash \fhi\}$.
\end{theorem}

\begin{proof}
 As already noted, we can assume $X={}^\omega 2$. For $x \in {}^\omega
 2$ let $S^x$  be defined as in the 
 previous section, so that the
 map which sends $x$ to $S^x$ is Borel and injective. Let $R'$ be the
 quasi-order on $X \times S_\infty$ defined by $(x,p) R' (y,q) \iff x
 R y$. It is clear that $R$ and $R'$ are Borel equivalent (as witnessed by the
 maps $x \mapsto (x,id)$ and $(x,p) \mapsto x$), hence it is enough
 to prove the theorem for $R'$.
Our plan is to find a Borel and invariant subset $Z$ of $Mod_\L$ and
a reduction of $R'$ to the embeddability relation
$\sqsubseteq$ with range $Z$,  and then use the well-known fact (due to Lopez-Escobar, see e.g.\   \cite[Theorem 16.8]{kechris}) that such a $Z$ must
coincide with $Mod_\fhi$ for some $\L_{\omega_1 \omega}$-sentence
$\fhi$. 

Using the same notation of the previous lemmas, consider the map $f$
which sends 
 $(x,p)$ to $G_{S^x,p}$, which is clearly a Borel (in fact, continuous) 
 map. First note that $f$ reduces $R'$ to the embedding relation
 $\sqsubseteq$, as
\[ (x,p) R' (y,q) \iff x R y \iff S^x \leq_{max} S^y \iff G_{S^x}
\sqsubseteq G_{S^y} \iff G_{S^x,p} \sqsubseteq G_{S^y,q}.\]
We now claim that $f$ is injective. Assume $(x,p) \neq (y,q)$: if $x
\neq y$ then $S^x \neq S^y$, and therefore by Lemma \ref{lemma1} we get
that $G_{S^x} \not\cong G_{S^y}$, which in turn implies that
$G_{S^x,p} \not\cong G_{S^y,q}$ as well (so that, in
particular, $G_{S^x,p}$ and $G_{S^y,q}$ must be different). If instead
$x=y$ but $p \neq q$, 
then by Lemma \ref{lemma2} we get $G_{S^x,p} \neq G_{S^x,q} =
G_{S^y,q}$ and hence we are done.

Since $X \times S_\infty$ is a Borel set and $f$ is Borel and injective, we
get that $f(X \times S_\infty) \subseteq Mod_{\mathcal{L}}$ is a Borel
set and that $f^{-1}$ is Borel as well. But $f(X \times S_\infty)$ is
clearly invariant under isomorphism, so $f(X \times S_\infty) =
Mod_\fhi$ for some $\mathcal{L}_{\omega_1\omega}$-sentence
$\fhi$. Since $f$ and $f^{-1}$ witness the 
Borel equivalence between $R'$ and embeddability on $Mod_\fhi$, this
concludes the  proof.
\end{proof}

\begin{remark}\label{remeffective}
There is an \emph{effective} version of Theorem \ref{theor1} (as well as of Theorem \ref{theor1'} and the corollaries below). Using the fact that any $\Sigma^1_1$ subset $A$ of the Cantor space ${}^\omega 2$ is the projection of a recursive (not necessarily pruned) tree on $2 \times \omega$, one can check that the proof of Corollary \ref{cornormalform} gives that a $\Sigma^1_1$ quasi-order $R$ on ${}^\omega 2$ is also the projection of a \emph{recursive} normal tree with all the requested properties, and this can in turn be used to check that, once we have chosen a suitable coding for the target ordered combinatorial tree (e.g.\ a coding similar to the one which will be explicitly given in Theorem \ref{theor2}), the function $f$ from ${}^\omega 2 \times S_\infty$ to $Mod_\L$ constructed in the previous proof is $\Delta^1_1$-recursive (in fact,  $\Sigma^0_1$-recursive). As $f$ is injective and has a $\Delta^1_1$ domain, we get from the effective version of the properties of Borel injective functions (see e.g.\ \cite[Exercise 4D.7]{mosch}) that ${\rm range}(f) \in \Delta^1_1$ and $f^{-1}$ is $\Delta^1_1$-recursive. By \cite[Theorem 3.14]{vandenboom}, any invariant  $\Delta^1_1$-subset of $Mod_\L$ is the class of models of some \emph{computable} infinitary formula, that is of a formula in the effective version of the infinitary logic $\L_{\omega_1 \omega}$ where countable conjunction and disjunction are allowed only on effectively enumerable sets of formulas. Therefore we have the following: for every $\Sigma^1_1$ quasi-order $R$ on ${}^\omega 2$ there is a computable infinitary formula $\fhi$ such that $R$ is $\Delta^1_1$-equivalent (in fact, $\Delta^1_1$-isomorphic) to embeddability on $Mod_\fhi$, where $\Delta^1_1$-equivalence is simply the effectivization of $\sim_B$. Such result can then be naturally extended to all $\Sigma^1_1$ quasi-orders defined on spaces which are $\Delta^1_1$-isomorphic to ${}^\omega 2$, that is to $\Sigma^1_1$ quasi-orders defined on \emph{recursively presented} Polish spaces.
\end{remark}

Now we will concentrate on other model-theoretic notions of morphism,
namely homomorphism and weak-homomorphism. For simplicity of notation, the definitions are given just for the language $\L$ under consideration in this section, but can clearly be generalized in a straightforward way to arbitrary languages.

\begin{defin}
If $G,G'$ are two $\mathcal{L}$-structures (where $\L$ is again the
language containing just the two binary relational symbols $P$ and
$Q$), we say that $G$ is 
\emph{homomorphic} to $G'$ if there is a function $h$ such that for
every $g_0,g_1$ in the domain of $G$, $g_0 P^G g_1 \iff h(g_0) P^{G'}
h(g_1)$ and $g_0 Q^G g_1 \iff h(g_0) Q^{G'} h(g_1)$ (such an $h$ will
be called a \emph{homomorphism} between $G$ and $G'$). 

Moreover, we say that $G$ is \emph{weakly-homomorphic} to $G'$ just in
case there is a function $h$ (called 
\emph{weak-homomorphism}) such that for
every $g_0,g_1$ in the domain of $G$, $g_0 P^G g_1 \imp h(g_0) P^{G'}
h(g_1)$ and $g_0 Q^G g_1 \imp h(g_0) Q^{G'} h(g_1)$.
\end{defin}

The relevance of the notion of
homomorphism between graphs is briefly described in
\cite{louros}. One should also note that embeddings are just
\emph{injective} homomorphisms.

\begin{theorem}\label{theor1'}
The relation of homomorphism (resp.\ weak-homomorphism) on $OCT$ is a
complete analytic quasi-order.
Moreover, if $R$ is an analytic quasi-order on a standard Borel space
$X$ then there is an $\mathcal{L}_{\omega_1\omega}$-sentence $\fhi$  
  such that $R$ is Borel equivalent to
 the relation of homomorphism (resp.\ weak-homomorphism)
on $Mod_{\fhi}$.
\end{theorem}

\begin{proof}
For the relation of homomorphism, one should simply note that since the
order of an (isomorphic copy of an) ordered combinatorial tree of the
form $G_T$ is reflexive and antisymmetric, each homomorphism
between $G_{S,p}$ and $G_{T,q}$ must be injective, i.e.\ homomorphisms and
embeddings coincide on $Mod_\fhi$ (where $G_{S,p}$, $G_{T,q}$ and
$Mod_\fhi$ are defined as in the proof of Theorem \ref{theor1}). 

For the relation of weak-homomorphism,
 first replace the order $\leq_T$ of $G_T$ with its strict part
 $<_T$ (denote such a structure by $\bar{G}_T$), where $G_T$ is
 defined as above. By 
 linearity of $\leq_T$, this implies that any weak-homomorphism $h$
 between $\bar{G}_S$ and $\bar{G}_T$  must be injective. But since the
  graph relation on any structure of the form $\bar{G}_T$ is connected, 
 $h$ must also be a homomorphism (hence an embedding): 
 for each pair of distinct
 elements $g,g' \in \bar{G}_S$ either $h(g) \nless_T h(g')$ or $h(g')
 \nless_T h(g)$, so if $h(g) <_T h(g')$ (which in particular implies
 $h(g) \neq h(g')$, and hence also $g \neq g'$) then $g <_S g'$ (otherwise by
 linearity of $\leq_S$ we would have $g' <_S g$ and hence $h(g') <_T
 h(g)$, a contradiction!). Moreover, let $g,g' \in \bar{G}_S$ be such
 that $h(g)$ and $h(g')$ are linked by the graph relation of
 $\bar{G}_T$, and let $g = g_0, g_1,  \dotsc, g_n=g'$ be the (unique)
 path in $\bar{G}_S$ which goes from $g$ to $g'$: since $h$ is a weak-homomorphism 
and is injective, $h(g_0), \dotsc, h(g_1)$ must be a
 path from $h(g)$ to $h(g')$ and if $n>1$ this would form a cycle
 because by hypothesis $h(g_0) = h(g)$ is linked to $h(g_n) =
 h(g')$ by the graph
 relation of $\bar{G}_T$, a contradiction with the absence of loops in
 $\bar{G}_T$! Hence
 $n=1$ and $g = g_0$ must be linked to $g' = g_1$ by the graph
 relation of $\bar{G}_S$.
Therefore, on (isomorphic copies of) structures of the form
$\bar{G}_S$, the notions of weak-homomorphism, homomorphism and
embedding coincide, and we can get the result by systematically
replacing  $G_T$ with  $\bar{G}_T$ in
Theorem
\ref{theorcomplete}, Lemma \ref{lemma1}, Lemma \ref{lemma2} and
Theorem \ref{theor1}.
\end{proof}

\begin{corollary}\label{cor10}
 Given a standard Borel space $X$, a binary relation $R$ on $X$ is 
an analytic quasi-order if and only if there is an
 $\mathcal{L}_{\omega_1 \omega}$-sentence $\fhi$  such that $R$ is
 Borel equivalent to the relation of embeddability (resp.\
 homomorphism, weak-homomorphism) on $Mod_\fhi$.
\end{corollary}

\begin{proof}
 One direction is given by Theorems \ref{theor1} and \ref{theor1'},
 while for the other direction just note that ``being an analytic
 quasi-order'' is downward closed with respect to Borel reducibility.
\end{proof}

\begin{corollary}\label{corequiv}
A binary relation $E$ on a standard Borel space $X$ is an analytic
 equivalence relation if and only if it is Borel equivalent to a
 bi-embeddability (resp.\ bi-homomorphism, bi-weak-homomorphism) relation.
\end{corollary}

\begin{proof}
 For the nontrivial direction, apply Corollary \ref{cor10} to the quasi-order $E$. 
\end{proof}

As observed in the introduction, what we have really shown is that given an analytic equivalence
relation $E$ on a standard Borel space $X$
there is an $\L_{\omega_1 \omega}$-sentence $\fhi$ and a \emph{Borel}
function $f \colon X \to Mod_\fhi$ which witnesses that $E$
is \emph{Borel isomorphic} to $\equiv$ on $Mod_\fhi$ and that $=$ on $X$ is \emph{Borel isomorphic} to $\cong$ on $Mod_\fhi$. Moreover, if $X$
is homeomorphic to ${}^\omega 2$, by Remark \ref{remarkoc} and the
observation following Theorem \ref{theomax} we get that $f$ has the further nice property
of being a \emph{homeomorphism} on its range (and hence is a topological embedding, a very simple function).

\begin{remark}
\begin{enumerate}[a)]
 \item Note that in Corollary \ref{corequiv} we get (by the symmetry of $E$) that there is an $\L_{\omega_1 \omega}$-sentence $\fhi$ such that $E \sim_B {{\sqsubseteq} \restriction {Mod_\fhi}}$, but with the further property that embeddability and bi-embeddability coincide on $Mod_\fhi$, that is if $x,y \in Mod_\fhi$ are such that $x \sqsubseteq y$ then automatically $y \sqsubseteq x$ as well.

\item  Suppose that $E$ in Corollary \ref{corequiv} is $\cong$ on $Mod_\psi$, $\psi$ an $\L_{\omega_1 \omega}$-sentence of some countable language $\L$. It is worth nothing that in general the sentence $\fhi$ such that $(Mod_\psi, \cong) \sim_B (Mod_\fhi, \equiv)$  obtained by Corollary \ref{cor10} has in general very little in common with $\psi$, even if they share the same language. 
\end{enumerate}
\end{remark}

\subsection{Embeddings in Analysis}\label{sectionanalysis}

The following version of Theorem \ref{theor1} gives some applications
in analysis, but it is also interesting \emph{per se} as it shows that
one can replace the language $\L$ with two different binary symbols
with the graph language. 

\begin{theorem}\label{theor2}
 If $R$ is an analytic quasi-order on a standard Borel space $X$ then there is
 a sentence $\psi$ of $\mathcal{L}'_{\omega_1 \omega}$ (where
 $\mathcal{L}'$ is the graph language with just one binary relational
 symbol) such that 
 $R$ is Borel equivalent to embeddability on $Mod_\psi$ (in particular,
 $Mod_\psi$ is the collection of ``ordinary'' graphs which satisfy
 $\psi$).
\end{theorem}

\begin{proof}
 The proof is a modification of the argument used in
 Theorem \ref{theor1}. Given a normal tree $T$ on $2 \times \omega$,
 we will define a new combinatorial tree\footnote{Clearly in this proof we could also use \emph{rooted} combinatorial trees instead of combinatorial trees: this will be used in Corollary \ref{corcompact}.} $G'_T$ (without any order),
 prove that \cite[Theorem 3.1]{louros} still holds when
 replacing $G_T$ with $G'_T$, and then slightly modify the
 argument used in Theorem \ref{theor1} to get the new result. 
 We can assume again that $X = {}^\omega 2$.
First define $G'_T$: given a normal tree $T \in \mathcal{T}$, let
$G_T$ be defined as in the proof of 
 \cite[Theorem 3.1]{louros}, that is as in the proof of Theorem
\ref{theor1} but without the order relation. Let $\# \colon \seqo \to
\omega$ be any 
bijection. Now for every $s 
\in \seqo$, adjoin vertices $s^+$, $s^{++}$ and $(s^{++}, i^k)$
for $ i \leq \#s +2$ and $0 \neq k \in \omega$, and link $s^+$ to both
$s$ and $s^{++}$, 
$(s^{++},i)$ to $s^{++}$, and $(s^{++},i^k)$ to
$(s^{++},i^{k+1})$. This concludes the definition of $G'_T$.

Now it is
easy to see how to reprove \cite[Theorem 3.1]{louros}. For one
direction, if $f$ is an injective witness of $S \leq_{max} T$ such that $\# s \leq \# f(s)$ for every $s \in \seqo$ (the construction of such a witness from an arbitrary one is easy and is left to the reader),
then we can define the embedding $g$ from $G'_S$ to $G'_T$ by sending $s$ to
$f(s)$, $s^*$ to $f(s)^*$, $s^+$ to $f(s)^+$, $s^{++}$ to $f(s)^{++}$,
$(s^{++},i^k)$ to $(f(s)^{++}, i^k)$, and $(u,s,x)$ to
$(u,f(s),x)$. For the other direction, note that all the points in
$G'_S$ have valence $\leq 2$ except for those of the form $s \in
\seqo$ (which have valence $\omega$), $s^{++}$ (which have valence
$\#s +4$), and $(u,s,0^{2 \theta(u)+2})$ (which have valence
$3$). Moreover, the distance from $s$ to $s^{++}$ is always $2$ and
vertices of the form $(u,s,0^{2 \theta(u)+2})$ are the unique vertices which
have valence $\geq 3$ and distance $2 \theta(u)+3$ from
$s$. Using all these facts, together with those about distances and
valences in $G_0$, we reach the conclusion that if $g$ is an embedding
of $G'_S$ in $G'_T$, then $f = g \restriction \seqo$ is such that
${\rm range}(f) \subseteq \seqo$, $f(\emptyset) = \emptyset$, and $f$
witnesses $S \leq_{max} T$ (the proof being exactly the same as in
\cite[Theorem 3.1]{louros}).

As before, each structure of the kind $G'_T$ needs to be 
Borel-in-$T$ coded into a structure $\hat{G}'_T$ with domain $\omega$ to fit
the official definition of countable $\L'$-structure: this can be done in
several ways, but in our case we need 
to specify a particular coding, at least for those $G'_T$ coming from an
infinite $T \in \mathcal{T}$. Let $\langle \cdot , \cdot \rangle$ be
any bijection from $\omega \times \omega$ to $\omega$. Then code 
vertices of the form $s$, $s^*$, $s^+$ and $s^{++}$ by $\langle 0, \#
s \rangle$, $\langle 1, \# s \rangle$, $\langle 2, \# s \rangle$ and
$\langle 3, \# s \rangle$, respectively. Now let $\eta_T$ be an
enumeration of $T$ such that $\eta_T((u,s)) \leq \eta_T ((v,t))$ if
and only if either $\# s < \# t$, or else $\# s = \# t$ and $u \leq_{lex} v$
($\eta_T$ is well-defined as in the latter case $s = t$, and $(u,s),(v,t)
\in T$ implies 
$|u| = |s| = |t| = |v|$), and code each vertex of the form $(u,s,x) \in
G'_T$, where $x 
\subseteq 0^{2 \theta(u) + 2}$, in an element of the form $\langle 4,
n \rangle$ in such a way that for two such vertices $(u,s,x)$ and
$(v,t,y)$ and corresponding codings $\langle 4,n \rangle$ and $\langle
4, m \rangle$ one has $n \leq m$ if and only if either $\eta_T((u,s))
< \eta_T((v,t))$ or $\eta_T((u,s)) = \eta_T((v,t))$ and $x \subseteq y$.
Finally, let 
$\eta_0$ be an
enumeration of $\langle (s,i) \mid {s \in \seqo} \wedge {i \leq \# s + 2} \rangle$ such that
$\eta_0((s,i)) \leq \eta_0((t,j))$ if and only if either $\# s < \# t$,
or else $\# s = \# t$ and $i \leq j$, and define $\pi_0 \colon \omega \to
\seqo$ and $\pi_1 \colon \omega \to \omega$ in such a way that
$\eta_0(\pi_0(n), \pi_1(n)) = n$. Now for $k \in \omega$ code vertices
of $G'_T$  of the form 
$(s^{++}, i^{k+1})$, $(u,s,0^{2 \theta(u) +2} {}^\smallfrown 0
  {}^\smallfrown 0^k)$ and $(u,s,0^{2 \theta(u) +2} {}^\smallfrown 1
  {}^\smallfrown 0^k)$  into $\langle 3
  \eta_0((s,i)) + 5, k \rangle$, $\langle 3
  \eta_T((u,s)) +6, k \rangle$ and $\langle 3
  \eta_T((u,s)) + 7, k \rangle$, respectively. This finishes the
  coding, and we will always identify $G'_T$ with its coded version
  $\hat{G}'_T$.

   Now consider the closed subgroup $H \subseteq S_\infty$ given by those
   bijections $p$ such that $p(\langle n,k \rangle) = \langle m,k
   \rangle$,  where $m$ depends only on $n$ (that is, if $p(\langle n , 0 \rangle) = \langle m, k \rangle$ then $p(\langle n,k \rangle) = \langle m,k \rangle$ for every $k \in \omega$) and the following conditions hold:
\begin{itemize}
 \item  $n = m$ if $n \leq 4$
\item $m = 3
   \eta_0((\pi_0(j),l))+5$ if $n = 3j+5$
\item   $m = n$ or $m = n+1$ if $n
   = 3j+6$
\item $m = n$ or $m = n-1$ if $n = 3j+7$.
\end{itemize}
 Notice that $H$ consists
   exactly of all automorphisms of $\hat{G}'_T = G'_T$ for some/every
   infinite $T \in \mathcal{T}$.  
    By a theorem of Burgess (see
   e.g.\  \cite[Theorem 12.17]{kechris}), there is a Borel selector for
   the equivalence relation on $S_\infty$ whose classes are the (left)
   cosets of $H$. Let $Y$ be the corresponding Borel transversal, and
   consider the quasi-order $R'$ defined on $X \times Y$ by letting
   $(x,p) R' (y,q) \iff x R y$. Let $f \colon X \times Y \to CT$ be the
   Borel (in fact, continuous) function which sends $(x,p)$ to $G'_{S^x,p} =
   j_{\L'}(p,G'_{S^x})$, where $S^x$ is defined as in 
   Section \ref{sectionpreliminaries}: it is immediate to check as before that $f$ is a reduction of
   $R'$ to $\sqsubseteq_{CT}$, so it is enough to show that $f$ is
   injective and that its range is invariant under isomorphism (this
   allows to conclude our proof as in Theorem \ref{theor1}).  

 The first claim (injectivity of $f$) follows from the fact that each
 isomorphism $i$ 
 between combinatorial trees of the form $G'_S$ and $G'_T$ (for $S,T \in
 \mathcal{T}$ infinite, as is the case if they are of the form
 $S^x$ because of the reflexivity of $R$) must belong to $H$:
 granting this, one should simply 
 note that $f((x,p)) = f((y,q))$ implies that $q^{-1} \circ p \in H$, so
 that $p=q$ (as they belongs to the same left coset of $H$ and both are in $Y$) and hence
 $G_{S^x} = G_{S^y}$, which in turn implies $x=y$ by injectivity of the
 map $x \mapsto S^x$. To prove the above statement, first note that
 since $i$ must preserve both distances and valences,  $i(s^{++}) =
 s^{++}$ because $s^{++}$ is the unique vertex (both
 in $G'_S$ and $G'_T$) with valence $\# s+4$, and therefore $i(s)=s$
 because $s$ is the unique vertex of valence $\omega$ with distance $2$
 from $s^{++}$. As in the proof of
 Lemma  \ref{lemma1}, this implies that $S = T$ (as $G_T$ and $G'_T$
 share the same vertices of valence $3$), and hence
 $G'_S = G'_T$: but this means $i \in H$ as required.

Finally, for the second claim (invariance of ${\rm range}(f)$) it
suffices to show that ${\rm 
  range}(f)$ is the saturation under isomorphism of the set $\{
G'_{S^x} \mid x \in X \}$, so consider a structure of the form
$j_{\L'}(p,G'_{S^x})$ for $x \in X$ and $p \in S_\infty$ (the other
inclusion is obvious). Let $q \in
Y$ be in the same (left) coset of $p$ with respect to $H$, so that $q
= p \circ h$ for some $h \in H$: then 
\[ j_{\L'}(p,G'_{S^x}) = j_{\L'}(q \circ h^{-1}, G'_{S^x}) =
j_{\L'}(q,j_{\L'}(h^{-1},G'_{S^x})) = j_{\L'}(q, G'_{S^x}) = f((x,q)),\]
since $h^{-1} \in H$ is necessarily an automorphism of $G'_{s^x}$.
\end{proof}

Clearly Theorem \ref{theor2} implies Theorem
\ref{theor1} as each combinatorial tree can be identified with the
ordered combinatorial tree with same graph relation and empty order,
and this identification preserves (closure under) isomorphisms and
embeddings. However, it seems that Theorem \ref{theor2} is really much
stronger than Theorem \ref{theor1} --- see the discussion in Section
\ref{sectionvaught}. 

The construction above allows us also to show that the relation
of homomorphism on graphs is a complete analytic quasi-order (a fact
already noted in  \cite[Theorem 3.5]{louros}), and that 
for each analytic quasi-order $R$ on $X$ there is an
$\mathcal{L'}_{\omega_1\omega}$-sentence $\psi$ such that $R$ is
Borel equivalent to the relation of homomorphism on $Mod_\psi$. This
follows from the next proposition and the fact the none of the
combinatorial trees
involved in the proof of Theorem \ref{theor2} have vertices of
valence $1$.

 \begin{proposition}\label{prophomo}
 Assume that $G$ is a combinatorial tree such that in $G$ there is no pair
 of vertices of valence $1$  with distance $2$ from each other, and
 $G'$ is an arbitrary graph. Then 
 every homomorphism from $G$ to $G'$ is injective (hence an embedding).
 \end{proposition}
 
 \begin{proof}
   It suffices to prove that if $h \colon G \to G'$ is a homomorphism
   and $g_0,g_1$ are distinct vertices of $G$ such that $h(g_0) =
   h(g_1)$ then these 
   vertices have both valence $1$ and distance $2$ from each
   other. This is an easy 
   consequence of the next claim.

   \begin{claim}\label{claimepi}
     Let $g_0, \dotsc , g_n$ be a chain of vertices in $G$ with $n \geq
     3$. Then $h(g_i) \neq h(g_j)$ for distinct $i,j \leq n$.
   \end{claim}
 
   \begin{proof}[Proof of the claim]
 By induction on $n \geq 3$. If $n=3$ first we have that $h(g_i) \neq
 h(g_{i+1})$ ($i \leq 2$) because $g_i$ is linked to $g_{i+1}$ in
 $G$. Then $h(g_0) \neq h(g_2)$ because otherwise $h(g_3)$ would be
 linked to $h(g_0)$ and therefore $g_3$ would be linked to $g_0$,
 a contradiction with the acyclicity of $G$. The same argument (using
 $g_1$ instead of $g_0$ in the second case) shows that $h(g_1) \neq h(g_3)$ and
 $h(g_0) \neq h(g_3)$.
 
 For the inductive step, 
 consider a chain $g_0, \dotsc, g_{n+1}$: since the claim must hold for
 both the chains $g_0, \dotsc, g_n$ and $g_1, \dotsc, g_{n+1}$, we need
 only to check that $h(g_0) \neq h(g_{n+1})$. But 
 $h(g_0) = h(g_{n+1})$ would contradict the acyclicity of $G$ again
 (since it implies that $g_1$ is linked to $g_{n+1}$), 
 hence we are done.
 \renewcommand{\qedsymbol}{$\square$ \textit{Claim}}
   \end{proof}
 
 \end{proof}

\begin{remark}\label{remweakhomo}
\begin{enumerate}[a)]
 \item Despite the previous result, we should note that the construction
given in Theorem \ref{theor2}  cannot be used to prove the 
analogous statement about the relation of weak-homomorphism on
combinatorial trees, as any two such trees are always
bi-weak-homomorphic. To see this it is enough to show that any
combinatorial tree $G$ is indeed bi-weak-homomorphic to the
combinatorial tree $\bar{G}$ on $\omega$ in which $n<m$ are linked
just in case $m=n+1$. In fact, choose any vertex $g_0$ of $G$: the map
which sends an arbitrary vertex of $G$ to its distance from $g_0$ is a
weak-homomorphism of $G$ into $\bar{G}$. Conversely, choose vertices
$g_0, g_1$ in $G$ such that there is an edge between them: the map
which sends $2k+i$ to $g_i$ (for $i=0,1$ and $k \in \omega$) is a
weak-homomorphism of $\bar{G}$ into $G$. 

\item A different proof of Theorem \ref{theor2} can be given using an argument similar to the one of Theorem \ref{theor1}: in fact it is possible to define for each normal tree $T$ on $2 \times \omega$ a combinatorial tree $G^+_T$ such that $S \neq T \imp G^+_S \not\cong G^+_T$ and each $G^+_T$ has no nontrivial automorphism, and then prove the desired results as in the proof of Theorem \ref{theor1} but using these last properties instead of Lemmas \ref{lemma1} and \ref{lemma2}. This alternative proof allows one to get also an effective version of Theorem \ref{theor2} analogous to the effective version of Theorem \ref{theor1} provided in Remark \ref{remeffective}, and is in a sense simpler than the one we gave above. However, we get the rigidity of $G^+_T$ by truncating at different heights those branches (i.e.\ maximal paths which start from $\emptyset$) which are not distinguishable in terms of their distance from all other branches of the tree (where``distance'' refers here to the distance $d_T$ defined in the proof of Corollary \ref{corultrametric}), that is e.g.\ the branches determined by the nodes $(s^{++},i)$ for different $i$'s. Such distinction would obviously get lost when passing to the corresponding ultrametric Polish space $U_T$ as defined in Corollary \ref{corultrametric}: this would result in a complication of the proof of Corollary \ref{corultrametric}, as we then should prove that any isometry between $U_S$ and $U_T$ can still be converted in an isomorphism between $G^+_S$ and $G^+_T$ (it is no more true that any isometry between $U_S$ and $U_T$ ``directly'' induces the unique isomorphism between $G^+_S$ and $G^+_T$, as it can now mix branches which are truncated at different heights). Moreover, the proof of Theorem \ref{theor2} we chose has the further advantage of introducing in a simple way a technique which will be used (in a more complicated way) in the proof Theorem \ref{theorVCcomplete}.

\item Combining the variant suggested above with \cite[Theorem 3.3]{louros}, one gets that the combinatorial trees used  in Theorem \ref{theor2} can be substituted by countable partial orders or countable lattices (viewed as partial orders, or even viewed as countable lattices, that is as structures in the language $\L''$ containing two binary function symbols and satisfying the axioms of lattices). In fact, in \cite[Theorem 3.3]{louros} a map $G \mapsto \leq_G$ from combinatorial trees on $\omega$ to countable partial orders (or to lattices) is constructed, and inspecting its proof one gets that each nontrivial (that is, different from the identity) isomorphism between $\leq_G$ and $\leq_H$ can be turned into a nontrivial isomorphism between $G$ and $H$: therefore we get that the map $T \mapsto \leq_{G^+_T}$ is such that $S \neq T \imp \leq_{G^+_S} \not\cong \leq_{G^+_T}$ and $\leq_{G^+_T}$ has only trivial automorphisms, hence we can conlude the proof as in Theorem \ref{theor1} again.
\end{enumerate}
\end{remark}

The main advantage of using Theorem \ref{theor2} is that we can get
several applications in analysis as corollaries  
(we could not have used Theorem \ref{theor1} because of
the orderings). First consider the class $\mathscr{D}$ of
\emph{discrete Polish metric spaces} $(\mathcal{X},d)$ (i.e.\ discrete
separable complete metric spaces): since any discrete separable
topological space is countable, we can identify each of them as a
space on $\omega$, i.e.\ we can put $\mathcal{X} = \omega$. Granting
this identification we have the following 
result (recall that an isometric
embedding is simply an injection between metric spaces which preserves
distances, while an isometry is just a surjective isometric embedding). 

\begin{corollary}\label{cordiscrete}
 If $R$ is an analytic quasi-order on a standard Borel space $X$
 then there is a Borel class $\mathcal{C} \subseteq \mathscr{D}$ closed
 under isometry such that $R$ is 
  Borel equivalent to the relation of isometric embeddability on
 $\mathcal{C}$.
\end{corollary}

\begin{proof}
With the notation established in the proof of Theorem \ref{theor2},
consider  each graph of the form $G'_{S,p}$ as a discrete Polish 
 space, where the distance is the geodesic distance on
 $G'_{S,p}$. Since one can recover such a distance from the graph
 structure and, conversely, the graph structure from the distance, it
 is clear that embeddings must correspond exactly to isometric
 embeddings and isomorphisms to isometries, therefore the result
 follows immediately from Theorem \ref{theor2}.
 \end{proof}

If one wants to deal with uncountable Polish spaces, the standard
procedure is to identify (up to isometry) each such space with a
closed subset of the Polish Urysohn space $\mathbb{U}$ (where the class
$F(\mathbb{U})$ of closed subsets of $\mathbb{U}$ is endowed with the
Effros-Borel topology), and then consider the analytic relation of
isometric embeddability on $F(\mathbb{U})$. If we now look at the Borel
class $\mathscr{U} \subseteq F(\mathbb{U})$ of ultrametric  Polish
spaces (i.e.\ metric Polish spaces whose distance $d$ is such that $d(x,y) \leq \max \{ d(x,z),d(z,y) \}$) we get the following: 

\begin{corollary}\label{corultrametric}
 For every analytic quasi-order $R$ on a standard Borel space $X$
 there is a Borel class $\mathcal{C} \subseteq \mathscr{U}$ of
 pairwise non-isometric ultrametric Polish spaces such that $R$ is
 Borel equivalent to the relation of isometric embeddability on
 $\mathcal{C}$. 
\end{corollary}

\begin{proof}
 Identify each element of the form $G'_T$ (for $T \in \mathcal{T}$)
 with the set $U_T$ of all maximal paths $\alpha$ starting from
 $\emptyset$, equipped with the distance $d_T$ defined by
 $d_T(\alpha_0, \alpha_1) = 0$ 
 if $\alpha_0 = \alpha_1$, and $d_T(\alpha_0, \alpha_1) = 2^{-n}$ if
 the set of vertices belonging to both $\alpha_0$ and $\alpha_1$ has
 cardinality $n$. It is clear that $U_T = (U_T, d_T)$ is an
 ultrametric Polish space, and that $G'_S \cong G'_T$ if and only if
 $U_S$ is isometric to $U_T$: this is because any isomorphism (resp.\
 embedding) between $G'_S$ and $G'_T$ can be canonically converted into
 an isometry (resp.\ isometric embedding) between $U_S$ and $U_T$, and
 \emph{vice-versa}. Now consider the Borel map $g$ which sends $x
 \in X$ to the isometric copy of $U_{S^x}$ in $\mathbb{U}$ (which is
 an element of $\mathscr{U}$): it is clearly injective, and by Theorem
 \ref{theor2} has the further 
 property that $x \neq y$ implies that $g(x)$ and $g(y)$ are not
 isometric; therefore it is enough to put $\mathcal{C} = {\rm
   range}(g)$. 
\end{proof}

Now we turn our attention to continuous embeddability. Each compact
metrizable space can be identified, up to homeomorphism, with an element
of the space $K(I)$, the space of all compact subspaces of the
Hilbert cube $I = [0,1]^\omega$ (with its Hausdorff topology). If we
consider the relations of, respectively, continuous embeddability
(given by injective continuous maps) and homeomorphism between
elements of $K(I)$, we get, respectively, an analytic quasi-order and
an analytic equivalence relation: also in this case, Theorem
\ref{theor2} leads to the following corollary\footnote{We have proved
  this result for the class $K([0,1]^2)$ of compact subsets of
  $[0,1]^2$, but since any such space can be naturally identified, up
  to homeomorphism, with an element of $K(I)$ the corollary holds with
  $K([0,1]^2)$ replaced by $K(I)$ as well.  Nevertheless, we cannot
  replace $K([0,1]^2)$ by $K([0,1])$ because, as already noted in
  \cite{louros}, the notion of continuous embeddability on $K([0,1])$
  gives just a pre-well-ordering of type $\omega_1 +2$.}. 

\begin{corollary}\label{corcompact}
 For every analytic quasi-order $R$ on a standard Borel space $X$
 there is a Borel class $\mathcal{C} \subseteq K([0,1]^2)$ of
 pairwise non-homeomorphic compact metrizable spaces such that $R$ is
 Borel equivalent to the relation of continuous embeddability on
 $\mathcal{C}$. 
\end{corollary}

\begin{proof}
 Consider the construction
 given in  \cite[Theorem 4.5]{louros} which defines a Borel map $S
 \mapsto K_S$ from rooted combinatorial trees to $K(I)$. Analyzing that proof, it is clear that for
 distinct rooted combinatorial trees $S,T$ one gets $S \cong T$ if and
 only if $K_S$ and $K_T$ are homeomorphic. Therefore the Borel map $g$
 which sends $x \in X$ to $K_{G'_{S^x}}$ is such that $g(x)$ is non-homeomorphic 
to $g(y)$ for distinct $x,y \in X$. Taking $\mathcal{C}
 = {\rm range}(g)$ we get the result. 
\end{proof}

Finally, we look at separable Banach spaces. Any such space is
linearly isometric to a closed subspace of $C([0,1])$ with
the sup norm, so the class of separable Banach spaces can be
identified with the Borel subset $\mathscr{B} \subseteq
F(C([0,1]))$ of all closed linear subspaces of
$C([0,1])$ (which is a standard Borel space). A function
between two separable Banach spaces $B$ and $B'$ is said to be a \emph{linear
  isometric embedding} (resp.\ \emph{linear isometry}) if it is linear
and norm-preserving (resp.\ linear, norm-preserving and onto): the
corresponding relations of 
linear isometric embeddability and linear isometry on $\mathscr{B}$
are, respectively, an analytic quasi-order and an analytic equivalence
relation. As noted in \cite{louros}, recent results by Godefroy and
Kalton show that on $\mathscr{B}$ these two relations coincide with
isometric embeddability and isometry, respectively. 

\begin{corollary}\label{corbanach}
 For every analytic quasi-order $R$ on a standard Borel space $X$
 there is a Borel class $\mathcal{C} \subseteq \mathscr{B}$ of
 pairwise non-linear isometric (resp.\ non-isometric) separable Banach
 spaces isomorphic to $c_0$ such that $R$ is Borel equivalent to the
 relation of linear isometric embeddability (resp.\ isometric
 embeddability) on $\mathcal{C}$. 
\end{corollary}

\begin{proof}
 Given a combinatorial tree $G$, consider the construction of the
 space $(c_0, \|\ \|_G)$ given in the proof of \cite[Theorem 4.6]{louros}. Since that proof shows that $G_0 \cong G_1$ if and
 only if $(c_0, \|\ \|_{G_0})$ and $(c_0, \|\ \|_{G_1})$ are linear
 isometric, the map $g$ which sends $x \in X$ to $(c_0, \|\
 \|_{G'_{S^x}})$ is Borel and strongly injective (in the sense that if
 $x\neq y \in X$ then $(c_0, \|\ \|_{G_{S^x}})$ and $(c_0, \|\
 \|_{G_{S^y}})$ are not linear isometric) by the proof of Theorem
 \ref{theor2} again. Therefore it is again enough to put $\mathcal{C}
 = {\rm range}(g)$. 
\end{proof}

Obviously, all the previous corollaries can be naturally translated
into the context of analytic equivalence relations. 

\begin{corollary}
 Let $E$ be a binary relation on a standard Borel space $X$. Then the
 following are equivalent: 
\begin{enumerate}
\item $E$ is an analytic equivalence relation on $X$; 
\item there is a Borel class $\mathcal{C} \subseteq \mathscr{D}$ of discrete
  Polish metric spaces closed 
 under isometry such that $E$ is Borel equivalent to the relation of
 isometric bi-embeddabbility on $\mathcal{C}$;
\item  there is a Borel class $\mathcal{C} \subseteq \mathscr{U}$ of
 pairwise non-isometric ultrametric Polish spaces such that $E$ is
 Borel equivalent to the relation of isometric bi-embeddability on
 $\mathcal{C}$;
\item there is a Borel class $\mathcal{C} \subseteq K(I)$ (or even
  just $\mathcal{C} \subseteq K([0,1]^2)$) of 
 pairwise non-homeomorphic compact metrizable spaces such that $E$ is
 Borel equivalent to the relation of continuous bi-embeddability on
 $\mathcal{C}$;
\item  there is a Borel class $\mathcal{C} \subseteq \mathscr{B}$ of
 pairwise non-linear isometric (resp.\ non-isometric) separable Banach
 spaces isomorphic to $c_0$ such that $E$ is Borel equivalent to the
 relation of linear isometric bi-embeddability (resp.\ isometric
 bi-embeddability) on $\mathcal{C}$. 
\end{enumerate}

\end{corollary}

\subsection{Actions of groups and monoids} \label{sectionmonoid} 

 A special kind of analytic equivalence relations is one which is
 induced by the continuous (resp.\ Borel) action of a group:
 given a Polish space $X$, a Polish group $G$ (that is, a group
 equipped with a Polish topology such that the map $(g,h) \mapsto g
 h^{-1}$ is continuous), and an action $a$ of $G$ on $X$ (that is a
 function $a \colon G \times X \to X$ such that $a(e,x) = x$ and
 $a(g,(a(h,x))) = a(g  h,x)$ for every $g,h \in G$ and $x \in X$), for each $x,y \in X$ we
 put $x E^G_X y \iff \exists g \in G (a(g,x)=y)$, 
 and it is easy to check that if $a$ is a continuous (or just
 Borel) function that $E^G_X$ is an analytic equivalence relation on
 $X$. Note however that such an analytic equivalence relation cannot
 be complete, as there are analytic equivalence relations, such as
 $E_1$, that are not  Borel reducible to it (see e.g.\ \cite[Theorem 8.2]{hjorth}).

The natural counterpart of Polish groups in the quasi-order context
are Polish  monoids, i.e.\ semigroups with  identity which are equipped
with a Polish topology such that the monoid operation is a continuous
function. However, it is not clear what should be the right
generalization of the notion of \emph{action}: in \cite{louros}, Louveau and Rosendal chose to
define the action of the monoid $G$ on the Polish space $X$
exactly as an action of a group, that is as a function $a
\colon G \times X \to X$ such that $a(e,x) = x$ and  $a(g,(a(h,x))) =
a(g  h,x)$ for every $g,h \in G$ and $x \in X$, and then proved in
\cite[Theorem 5.1]{louros} that there is a Polish monoid $G$ acting continuously on a
Polish space $X$, such that $E^G_X$ is complete for analytic
quasi-orders. There is however a difficulty with this notion of monoid action: if one looks at some natural quasi-order $R$
induced by some class of morphisms acting on $X$,  then the monoid $G$
consisting of such morphisms together with the composition operation
and the usual identity (and topologized in a natural way) should have
a ``natural'' action on $X$ inducing $R$, and
 this will not in general be true using the notion of action defined above. For 
example, consider the relation of embeddability on the space $X_G$ of
(codes for) countable graphs: in this case the collection of
morphisms is simply the Polish monoid $S_\infty^-$, the set of
all injective (not necessarily onto) functions from $\omega$ into
itself, topologized as a subspace of the Baire space. But there is no
natural function assigning to each pair $(p,x) \in S^-_\infty \times X_G 
$ a \emph{unique} $y \in X_G$ such that $x \sqsubseteq y$, as
there are too many  $y$ in which $x$ can be embedded. One option is to change the notion of monoid in such a way that each of its elements determines in a unique way the target graph, not only how $x$ is embedded into it. We have taken this option in the next 
section in the special case of embeddings on graphs; however, this approach
can be used just for some specific cases. This suggests that
the ``right'' definition of an action of a monoid on a Polish space
should be that of a relation (or, equivalently, of a multi-valued
function) rather than that of a function: here is our proposal. 

\begin{defin}\label{defaction}
 Let $X$ be a Polish space and $G$ a Polish monoid. An action $A$ of
 $G$ on $X$ is a relation $A \subseteq G \times X \times X$ such that
 for every $x,y,z \in X$ and $g,h \in G$ the following holds:  
\begin{enumerate}[i)]
 \item $(e,x,x) \in A$;
\item if $(h,x,y) \in A$ and $(g,y,z) \in A$ then $(gh,x,z) \in A$.
\end{enumerate}
 The action $A$ is said to be \emph{closed} (resp.\ \emph{Borel},
 \emph{analytic}) if it is closed (resp.\ Borel, analytic) as a
 subspace of the Polish space $G \times X \times X$. 
\end{defin}

Note that the usual group action can be identified (passing from
functions to their graph) with those monoid actions which turn out to
be functions from $G \times X$ into $X$, that is with those $A
\subseteq G \times X \times X$ such that for every $(g,x) \in G \times
X$ there is a unique $y \in X$ such that $(g,x,y) \in A$. 
Being closed (resp.\  Borel) for monoid actions is the
analogous of being continuous (resp.\  Borel) for group actions,
although the 
 analogy is not exact in the first case as there are functions with
 closed graphs which are not continuous\footnote{As an example of this fact, consider the function $f \colon \mathbb{R} \to \mathbb{R}$ such that $f(x)=0$ if $x=0$ and $f(x) = x^{-1}$ otherwise.}. In the case of
 group actions the distinction between Borel and analytic simply disappears,
 because a function is Borel if and only its graph is  Borel if
 and only if its graph is analytic if and only if it is analytic-measurable. 

For $A$ an action of $G$ on $X$ we put $x E^{G,A}_X y \iff \exists g
\in G ((g,x,y) \in A)$, and it is easy to check that  $E^{G,A}_X$ is a
quasi-order (by the properties of an action), and that if $A$ is analytic 
then $E^{G,A}_X$ is analytic: we will call $E^{G,A}_X$
the analytic quasi-order induced by the action $A$ of $G$ on $X$. 
Note that if $X$ is just a standard Borel
space (rather than Polish), then we can speak of closed actions on $X$
just in case we fix in advance some Polish topology on $X$ compatible
with its Borel structure, but we can unambiguously speak of Borel (or
analytic) actions, as these notions only depends on the Borel
structure of $X$. 

Now consider again the example of embeddings on the space of graphs
$X_G$: with our new definition, the monoid $S^-_\infty$ 
has a natural closed action on $X_G$, namely 
\[ A = \{ (p,x,y) \in A \mid \forall n,m (n R^{\mathcal{A}_x} m \iff
p(n) R^{\mathcal{A}_y} p(m)) \}\] 
 (where $\mathcal{A}_x$ and $\mathcal{A}_y$ are the graphs coded by
 $x$ and $y$, respectively), and it is clear that the quasi-order
 $E^{S^-_\infty,A}_{X_G}$ is just the relation of
 embeddability on $X_G$. 

\begin{theorem}
Each analytic quasi-order $R$ on a standard Borel space $X$ is
Borel equivalent to a quasi-order on ${}^\omega 2$ induced by a closed
action of $S^-_\infty$. Moreover, $R$ itself is induced by a Borel
action of $S^-_\infty$.  
\end{theorem}

\begin{proof}
As any standard Borel space $X$ is Borel-isomorphic to ${}^\omega 2$,
we just need to prove that any quasi-order $R'$ on ${}^\omega 2$ is
induced by a closed action $A'$ of $S^-_\infty$ on ${}^\omega 2$.  
 Let $A$ be the natural closed action $A$ of $S^-_\infty$ on the space
 of (codes for) graphs $X_G$, and let $f \colon {}^\omega 2 \times
 Y \to X_G$ be the reduction defined in the proof of Theorem
 \ref{theor2}, where $Y$ is the Borel transversal defined in
that proof. Clearly we can assume $id \in Y$. 
Since $f$ is a continuous function, so is $h \colon
 {}^\omega 2 \to X_G$ defined by $x \mapsto f((x,id))$: therefore the action
 $A' \subseteq S^-_\infty \times {}^\omega 2 \times {}^\omega 2$ given 
by $(p,x,y) \in A'
 \iff (p, h(x),h(y)) \in A$ is closed and clearly induces $R$ because
 $h$ is a reduction of $R'$ to $\sqsubseteq$ (which is
 $E^{S^-_\infty,A}_{X_G})$. 
\end{proof}

\section{Applications of the main results and techniques}\label{sectionvaught} 

We would like to explore in this section some applications of the methods and results obtained above.
In \cite{louros}, Louveau and Rosendal suggested that one analyze the possible relationships between bi-embeddability $\equiv_{\mathcal{C}}$ and isomorphism $\cong_\mathcal{C}$ on some $\L_{\omega_1 \omega}$-elementary class $\mathcal{C}$ of countable structures (that is on the class of countable  models of some $\L_{\omega_1 \omega}$-sentence, $\L$ some countable language). As they already noted, these relations can behave very differently. A trivial example is the class $\mathcal{C}$ of countable well-founded linear order of length bounded by some fixed $\alpha < \omega_1$: in this case $\equiv_{\mathcal{C}}$ and $\cong_\mathcal{C}$ coincide. In contrast, some deep results show that $\equiv_\mathcal{C}$ and $\cong_\mathcal{C}$ can be extremely far apart.

\begin{example}
 Let $\mathcal{C}$ be the collection of countable linear orders: then $\equiv_\mathcal{C}$
has only $\aleph_1$ classes (since by Laver's proof of Fra\"iss\'e conjecture $\sqsubseteq_{\mathcal{C}}$ is  a bqo), whereas $\cong_\mathcal{C}$ is $S_\infty$-complete (that is as complicated as it can be) by \cite{friedmanstanley}.
\end{example}

\begin{example}
If $\mathcal{C}$ is the collection of countable graphs (or of combinatorial trees, partial orders, lattices, and so on) then $\equiv_\mathcal{C}$ is a complete analytic equivalence relation by \cite{louros}, while $\cong_\mathcal{C}$ is just $S_\infty$-complete.
\end{example}

Louveau and Rosendal raised the question of whether one can increase these gaps, namely:

\begin{question}\label{questLR1}
 Is there an $\L_{\omega_1 \omega}$-elementary class $\mathcal{C}$ with $\cong_\mathcal{C}$  $S_\infty$-complete but $\equiv_\mathcal{C}$ with countably many classes?
\end{question}

\begin{question}\label{questLR2}
 Is there an $\L_{\omega_1 \omega}$-elementary class $\mathcal{C}$ with $\equiv_\mathcal{C}$   complete analytic but $\cong_\mathcal{C}$ not $S_\infty$-complete?
\end{question}
 In the same vein, a third natural question concerning some possible limitations to the method developed in \cite{louros} was asked as well: 
\begin{question}\label{questLR3}
  Is there an $\L_{\omega_1 \omega}$-elementary class $\mathcal{C}$ with $\equiv_\mathcal{C}$ complete analytic but $\sqsubseteq_\mathcal{C}$ not a complete analytic quasi-order?
\end{question}

The results obtained in the previous section can be used to answer Questions \ref{questLR2} and \ref{questLR3}: in fact, as already observed, Theorem \ref{theor1} and Theorem \ref{theor2} both show (for different languages $\L$) that for every analytic quasi-order $R$ there is an $\L_{\omega_1 \omega}$-elementary class $\mathcal{C}$ such that $R \sim_B \sqsubseteq_\mathcal{C}$ (hence ${{E_R} = {R \cap R^{-1}}} \sim_B {\equiv_\mathcal{C}}$) 
and ${=} \sim_B {\cong_\mathcal{C}}$. This result can be used to find e.g.\ $\L_{\omega_1 \omega}$-elementary classes $\mathcal{C}$ such that ${\equiv_\mathcal{C}} = {\cong_\mathcal{C}}$ (take $R$ to be the equality relation on ${}^\omega 2$) or such that 
${\equiv_\mathcal{C}} \neq {\cong_\mathcal{C}}$ but still ${\equiv_\mathcal{C}} \sim_B {\cong_\mathcal{C}}$ (take $R$ to be any equivalence relation on ${}^\omega 2$ Borel equivalent to the equality relation but such that at least one equivalence class has more than one element: then if $\mathcal{C}$ is the class resulting by the application of our result, $\equiv_{\mathcal{C}}$ will be strictly coarser than $\cong_\mathcal{C}$ but ${\cong_\mathcal{C}} \sim_B {=} \sim_B R = E_R \sim_B {\equiv_\mathcal{C}}$).

Another trivial application of the same result allows one to answer Question \ref{questLR2}: let $R$ be any complete analytic quasi-order, and let $\mathcal{C}$ be the class corresponding to $R$ with respect to the result quoted above. Then $\sqsubseteq_\mathcal{C}$ (and hence also $\equiv_\mathcal{C}$) must be complete analytic (in the corresponding classes of relations), but $\cong_\mathcal{C}$ is just smooth (in fact, Borel equivalent to the equality relation). The same kind of argument shows that we can realize any possible relationship between $\cong_\mathcal{C}$ and $\equiv_\mathcal{C}$ (such as ${\equiv_\mathcal{C}} <_B {\cong_\mathcal{C}}$, or ${\equiv_\mathcal{C}}$ and ${\cong_\mathcal{C}}$ Borel incomparable, and so on) as long as we are content to have ${\cong_\mathcal{C}} \sim_B {=}$. An interesting related problem would be to determine which are the possible pairs of degrees $(\cong_\mathcal{C}, \equiv_\mathcal{C})$ for $\mathcal{C}$ an $\L_{\omega_1 \omega}$-elementary class, but apart from the previous results and some obvious limitations this seems to be difficult.

Question \ref{questLR3} can be answered in the same vein as Question \ref{questLR2}: let $R$ be a  complete analytic equivalence relation (so that, in particular, $E_R = R$), and let $\mathcal{C}$ be the resulting class obtained with Theorem \ref{theor1} or Theorem \ref{theor2}. Then $\equiv_{\mathcal{C}}$ is clearly complete analytic (as ${\equiv_\mathcal{C}} \sim_B E_R = R$), but $\sqsubseteq_\mathcal{C}$ can not be complete as a quasi-order since it is an equivalence relation (in fact, as already observed, in this case we get ${\sqsubseteq_\mathcal{C}} = {\equiv_\mathcal{C}}$). 

Question \ref{questLR2} and \ref{questLR3} can  be modified in the following way: given an $\L_{\omega_1 \omega}$-elementary class $\mathcal{C}$, call the relation $\sqsubseteq_\mathcal{C}$ (resp.\ $\equiv_\mathcal{C}$) \emph{universal} if for every analytic quasi-order $R$ (resp.\ every analytic equivalence relation $E$) there is an $\L_{\omega_1 \omega}$-elementary class $\mathcal{C}' \subseteq \mathcal{C}$ such that $R \sim_B {\sqsubseteq_{\mathcal{C}'}}$ (resp.\ $E \sim_B {\equiv_{\mathcal{C}'}}$). Note that if one of $\sqsubseteq_\mathcal{C}$ or $\equiv_\mathcal{C}$ is universal then it must be also complete analytic in the corresponding class of analytic relations.
Moreover, as Corollary \ref{corequiv} shows, if $\mathcal{C}$ is such that $\sqsubseteq_\mathcal{C}$ is universal then $\equiv_\mathcal{C}$ must be universal as well, and Theorems \ref{theor1} and \ref{theor2} can be rephrased as ``there exists an $\L_{\omega_1 \omega}$-elementary class such that $\sqsubseteq_\mathcal{C}$ (and hence also $\equiv_\mathcal{C}$) is universal''. Here are the natural modifications of Question \ref{questLR2} and Question \ref{questLR3}:

\begin{question}\label{questLR2'}
 Is there an $\L_{\omega_1 \omega}$-elementary class $\mathcal{C}$ such that $\equiv_\mathcal{C}$ is universal but $\cong_\mathcal{C}$ not $S_\infty$-complete?
\end{question}

\begin{question}\label{questLR33'}
 Is there an $\L_{\omega_1 \omega}$-elementary class $\mathcal{C}$ with $\equiv_\mathcal{C}$ universal but $\sqsubseteq_\mathcal{C}$ not universal?
\end{question}

In both cases the answer is positive again, and it is even possible to have a single $\mathcal{C}$ such that $\equiv_\mathcal{C}$ is universal but $\cong_\mathcal{C}$ is not $S_\infty$-complete and $\sqsubseteq_\mathcal{C}$ is not a complete analytic quasi-order (hence, in particular, not universal): in fact, it is enough to consider an $\L_{\omega_1 \omega}$-elementary class $\mathcal{C}'$ such that $\equiv_{\mathcal{C}'}$ is universal and then apply Theorem \ref{theor1} or Theorem \ref{theor2} to such a relation in order to get an $\L_{\omega_1 \omega}$-elementary class $\mathcal{C}$ with the desired properties (one has just to check that $\equiv_\mathcal{C}$ is indeed universal because any Borel subset of $\mathcal{C}'$, hence in particular any $\L_{\omega_1 \omega}$-elementary subclass of $\mathcal{C}'$, will be mapped to an $\L_{\omega_1 \omega}$-elementary subclass of $\mathcal{C}$ by the construction given in the proofs of the quoted theorems).

The answer to Question \ref{questLR1} is of a different nature, and doesn't involve the results of this paper, but we put it here for the sake of completeness and because the following construction (due to H.\ Friedman and Stanley, see \cite{friedmanstanley}) will be used below for different applications. 
Let
  $\langle \cdot, \cdot \rangle \colon \omega \times \omega \to \omega$
 be any bijection: for $\emptyset \neq s \in \seqo$ define the \emph{relevant
   pair of $s$} to be $rp(s) = (s(n),s(m))$, where $n,m$ are such that
 $s$ has length $\langle n,m \rangle +1$. 
Given (a code for) any countable graph $x$,
 we will define a set-theoretical tree $T_x$ in the
 following way: the domain of $T_x$ is given by $\seqo \sqcup
 \omega$ (where $\sqcup$ means disjoint union) if $x$ is not the empty graph and by $\seqo$ otherwise,  and we order $\seqo$ with the inclusion relation. Finally, if $x$ contains at least one edge we adjoin to each element $\emptyset \neq s \in \seqo$ such that
 $rp(s)$ is a pair of linked vertices in $x$ a new terminal immediate successor taken from 
 $\omega$, in such a way that each natural number is the terminal successor of one
 (and only one) element of $\seqo$ (this can be done since for each $n,m \in \omega$ there are infinitely
 many $s \in \seqo$ for which $rp(s) = (n,m)$). As proved in \cite{friedmanstanley}, each
isomorphism between graphs $x$ and $y$ can be naturally ``lifted'' to a
permutation of $\seqo$, and then extended to an isomorphism of
$T_x$ and $T_y$. Conversely, an isomorphism $j$ between $T_x$ and
$T_y$ must send elements of $\seqo$ to elements of $\seqo$ of
the same length: therefore we can
reconstruct from $j$ an isomorphism $i$ between $x$ and $y$ using a back and forth argument. Now let $\mathcal{C}$ be the class of infinite set-theoretical countable trees of height $\omega$ such that each node is either terminal or has infinitely many immediate successor, and such that in the latter case at most one of those successors is terminal (we leave to the reader to show that there is an $\L_{\omega_1 \omega}$-sentence $\fhi$ in the language of trees $\L$ such that $\mathcal{C} = Mod_\fhi$): then each tree of the form $T_x$ belongs to $\mathcal{C}$, whence $\cong_\mathcal{C}$ is $S_\infty$-complete, but each element of $\mathcal{C}$ can be embedded in any other element of $\mathcal{C}$ since any set-theoretical tree of height $\omega$ can be embedded in $\seqo$, and the last condition defining $\mathcal{C}$ easily imply that a tree in $\mathcal{C}$ must contain an isomorphic copy of $\seqo$ as subtree. This proves that $\mathcal{C}$ is such that $\cong_\mathcal{C}$ is $S_\infty$-complete but $\equiv_\mathcal{C}$ has exactly $1$ equivalence class (hence, in particular, provides another situation in which ${\equiv_\mathcal{C}} <_B {\cong_\mathcal{C}}$, but with the further property that $\cong_\mathcal{C}$ is as complicated as it can be).\\

The following different application of the methods developed in the previous section can in particular be  intended as an
explanation of why we heuristically  feel that Theorem \ref{theor2}
is really a stronger result than Theorem \ref{theor1}: our explanation
is based on some connections with the Vaught's conjecture, so let us
first recall some definitions and results regarding that topic. Here
is the statement of (a version of)  Vaught's conjecture, which is
still  open. 

\begin{Vaughtconj}\label{Vaught}
 For every countable language $\L$ and every $\L_{\omega_1
   \omega}$-sentence $\fhi$, there are either at most countably many
 or perfectly many non-isomorphic countable models of $\fhi$.  
\end{Vaughtconj}

Vaught's conjecture can be used to measure the ``complexity'' of
various theories in the following way.
Given an $\L_{\omega_1 \omega}$-sentence $\fhi$, we say that $\fhi$
(or, equivalently, the collection of all its countable models
$Mod_\fhi$) \emph{strongly satisfies Vaught's conjecture}, if for
every $\L_{\omega_1\omega}$-sentence $\psi$ such that $\psi \imp \fhi$
we have that either there are at most countably many non-isomorphic
countable models of $\psi$, or else there are perfectly many
non-isomorphic countable models of $\psi$. This can be equivalently
rephrased by requiring that every Borel invariant (with respect to
isomorphism) subset $X$ of $Mod_\fhi$ has either contably many or
perfectly many $\cong$-classes. Note also that this is stronger than
just requiring that $\fhi$ has either countably many or perfectly many
non-isomorphic countable models. In what follows we confuse
(countable) $\L_{\omega_1 \omega}$-theories $T$ with the $\L_{\omega_1
  \omega}$-sentences given by the (infinite) conjunction of the sentences in
$T$. Rubin \cite{rubin, rubin2}, Miller \cite{miller} and Steel
\cite{steel} showed that the theories of, respectively, linear orders,
unary operations (or even just graphs such that each connected
component has only finitely many loops), and trees (partially ordered sets in which the set of predecessors of any element
is linearly ordered) strongly satisfy Vaught's conjecture. The
result of Miller, in particular, implies that also the theory of (rooted)
combinatorial trees strongly satisfies Vaught's conjecture. 

Theories which strongly satisfy Vaught's conjecture can be
regarded in this context as ``simple'' theories. On the opposite side
there are those theories which have the property that if they strongly
satisfy Vaught's conjecture then Vaught's conjecture \ref{Vaught}
holds: we call these theories \emph{Vaught's conjecture-complete}
(\emph{VC-complete} for short), and can be regarded in the present context
as the most complicated theories\footnote{There is an
  analogy between Vaught's conjecture and the problem $P = NP$ in
  theoretical computer science: theories correspond to decision problems, theories
  strongly satisfying Vaugh's conjecture correspond to decision
  problems belonging to $P$, and VC-complete theories correspond to
  $NP$-complete decision problems.}. It is a folklore result that
 the theory of graphs is VC-complete. We will now show that also
the theory of rooted ordered combinatorial trees $ROCT$ is VC-complete, and
therefore (unless Vaught's conjecture is proved!) more complicated than the theory of (rooted)
combinatorial trees: this shows that Theorem \ref{theor2} is in a
sense stronger than Theorem \ref{theor1}, because it proves that the
 ``universality'' property of $\sqsubseteq_{ROCT}$ expressed by Theorem \ref{theor1} is shared
also by the relation of embeddability on a ``simpler''
theory. 

\begin{theorem}\label{theorVCcomplete}
 The theory of (rooted) ordered combinatorial trees (whose ordering is given by an equivalence relation) is VC-complete.
\end{theorem}

\begin{proof}
The proof is a combination of the Friedman-Stanley's construction explained above with the technique developed in the proof of
Theorem \ref{theor2}. 
 Let $X_G$ and $X_{ROCT}$ be the set of (codes for) countable
 graphs and rooted ordered combinatorial trees\footnote{As it will be easy to check, we can also replace rooted ordered combinatorial trees with ordered combinatorial trees: in the construction of the $G_x$'s given below, instead of specifying $\emptyset$ as root, simply adjoin two new vertices $r$ and $r'$, and then link $r$ to both $\emptyset$ and $r'$. It is easy to check that any embedding between two such trees must again send $\emptyset$ into itself, so the rest of the argument can be carried out exactly in the same way.}, respectively. 
It is enough to show that for every Borel invariant $X
 \subseteq X_G$ there is a Borel invariant $Z \subseteq X_{ROCT}$ such
 that $(X,\cong) \sim_B (Z,\cong)$. Note that we
 can always assume that $X$ doesn't contain the empty graph. To each $x \in X$, associate the rooted ordered combinatorial tree $G_x$ defined on $\seqo \sqcup \omega$ (with $\emptyset$ as root) by stipulating that two elements of $G_x$ are adiacent just in case one of them is an immediate successor of the other in the tree $T_x$ defined above, and that two elements $s,t$ are in the order relation of $G_x$ if and only if 
\[ (s = t = \emptyset) \vee ({s,t \in \seqo \setminus \{ \emptyset \}} \wedge {rp(s) = rp(t)}) \vee (s,t \in \omega). \] 
Note that such order relation is
  an equivalence relation which moreover is
\emph{independent}
 from the graph $x$. If we choose to link natural numbers to
 sequences in a careful way, e.g.\ respecting any fixed order of $\seqo$
 of type $\omega$, then the map which sends $x$ to (the code of) $G_x$
 is Borel (in fact, continuous). 
By repeating the Friedman-Stanley's proof  about the trees $T_x$ sketched above, it is easy to check that this map reduces the isomorphism relation on
any Borel invariant  $X \subseteq X_G$ to the isomorphism relation on $X_{ROCT}$ (it is enough to check that the ``lifting'' of any isomorphism between graphs $x$ and $y$ to a permutation of $\seqo$ preserves the order relations of $G_x$ and $G_y$).

For simplicity of notation, identify $S_\infty$ with the group of
permutations of $\seqo \sqcup \omega$. Consider the following closed
subgroups of $S_\infty$: 
\begin{align*}
 H_1 = \{p \in S_\infty \mid & \forall s,t \in \seqo ({rp(s) =
   rp(t)} \iff {rp(p(s)) = rp(p(t))} \\ 
& \wedge \forall n \in \omega (p(n)=n)\}
\end{align*}
and
\[ H_2  = \{ p \in S_\infty \mid \forall s \in \seqo (p(s) = s)\}.\]
 By \cite[Theorem 12.17]{kechris} again, there is a Borel selector
 for the equivalence relation on $S_\infty$ whose classes are the
 (left) cosets of $H_1$. Let $Y$ be the corresponding Borel transversal,
 and consider the equivalence relation $E$ on $X \times H_2 \times Y$
 ($X$ as before) defined by $(x,p_1,q_1) E (y,p_2,q_2) \iff x \cong
 y$. Clearly $E$ is Borel equivalent to isomorphism on $X$, and the map $f \colon X \times H_2 \times Y \to X_{ROCT} \colon
 (x,p,q) \mapsto j_\L(q \circ p,G_x)$, $\L$ the language of rooted ordered
 combinatorial trees, is a continuous reduction of $E$ into $\cong$ on
 $X_{ROCT}$. As in the proof of Theorem \ref{theor2}, it is therefore
 enough to show that $f$ is injective and that its range $Z$ is
 invariant under isomorphism. 

Assume that $x,y \in X$, $p_1,p_2 \in H_2$ and $q_1,q_2 \in Y$ are
such that $f((x,p_1,q_1)) = f((y,p_2,q_2))$. Then $p_2^{-1} \circ
q_2^{-1} \circ q_1\circ p_1$ is an isomorphism between $G_x$ and $G_y$
and since it must respect their order relations and both $p_2^{-1}$
and $p_1$ are the identity on $T_\infty$, we must conclude that
$q_2^{-1} \circ q_1 \in H_1$: but then $q_1$ and $q_2$ are in the same
left coset of $H_1$, which means $q_1 = q_2$ (since $Y$ is a
transversal). This implies $x = y$ because $p_2^{-1} \circ
q_2^{-1} \circ q_1\circ p_1 = p_2^{-1} \circ p_1 \in H_2$ is an
isomorphism between $G_x$ and $G_y$, and hence we get also $p_2^{-1}
\circ p_1 = id$, that is $p_1 = p_2$. This proves the injectivity of
$f$. 

Finally we will prove that ${\rm range}(f)$ is the closure under
isomorphism of $\{ G_x \mid x \in X \}$.  
First note that if $h \in H_1$ and $x \in X$, then there is $y \in X$
and $p \in H_2$ such that $j_\L (p,G_y) = j_\L(h,G_x)$ (this is
because $X$ is invariant under isomorphism). 
Now choose $r \in S_\infty$ and $x \in X$, and let $q \in Y$ be in the
same left coset of $H_1$ as $r$, so that $r = q \circ h$ for some $h
\in H_1$. Let $y$ and $p$ be as in the observation above. 
Then $j_\L(q \circ p,G_y) = j_\L(q \circ h, G_x) = j_\L(r,G_x)$, so
that $f((y,p,q)) = j_\L(r,G_x)$ and we are done. 
\end{proof}

Note that an obvious modification of the previous proof gives that also the theory of ordered trees (that is of set-theoretical trees with an
 extra transitive relation on their nodes) is VC-complete. 

 Besides showing that \emph{ordered} (rooted) combinatorial trees have
 seemingly more ``universal'' properties  than combinatorial trees
 (namely, the fact of being VC-complete), Theorem
 \ref{theorVCcomplete} should also be compared with the well-known
 argument used to show that the theory of graphs is VC-complete. In
 that case, one makes a crucial use of \emph{infinitely many} loops: if
 one could find a similar argument which avoids loops (or uses only
 finitely many of them in each connected component), then applying
 Miller's result \cite{miller} one would get a proof of  Vaught's
 conjecture. Theorem \ref{theorVCcomplete} shows that we can
 completely avoid loops, but (as far as we know) we have to compensate
 for this with the addition of a new transitive relation.

\section{Open problems}\label{sectionopenproblems}

We collect in this section some open problems related to some of the
results presented in Section \ref{sectionmain}.

We say that a function $f$ is an \emph{epimorphism} between countable $\L$-structures $H$ and $G$ if it is a surjective homomorphism  of $H$ onto $G$, and that $f$ is a \emph{weak-epimorphism} if it is a surjective weak-homomorphism.
Clearly the relations ``$G \preceq H$ if and only if $G$ is the epimorphic image of $H$'' and ``$G \preceq_w H$ if and only if $G$ is the weak-epimorphic image of $H$'' are analytic quasi-orders, and  Camerlo \cite{camerlo} showed that $\preceq_w$ on countable graphs is complete for analytic quasi-orders. Therefore we have the following natural question.

\begin{question}\label{questionepi}
 Is it true that for every analytic quasi-order $R$ on a standard Borel space $X$ there is an $\L_{\omega_1 \omega}$-sentence $\fhi$ ($\L$ some countable language) such that $R$ is Borel equivalent to $\preceq$ (respectively, $\preceq_w$) when restricted to $Mod_\fhi$?
\end{question}

Let us point out that combinatorial graphs are unlikely to play any role in answering this question: Proposition \ref{prophomo} shows that almost every epimorphism between combinatorial trees (in particular, every epimorphism between the combinatorial trees used in our constructions) is actually an isomorphism, so that $\preceq$ on combinatorial trees is probably more or less as complex as $\cong$ on combinatorial trees, while $\preceq_w$ on combinatorial trees seems to be a very simple relation, as it can be shown e.g.\ that any combinatorial tree with unbounded diameter is bi-weak-epimorphic to the combinatorial tree $\bar{G}$ defined in Remark \ref{remweakhomo}. It could be that the combination of the construction given in \cite{camerlo}  with the techniques developed in this paper would yield an answer to Question \ref{questionepi}, but we leave this possibility for future research. 

A similar questions can be asked about the relation of elementary embeddability between countable $\L$-structures, although (as far as we know) it is not even known if the corresponding quasi-order is complete analytic.

\begin{question}
Is there any $\L_{\omega_1 \omega}$-sentence $\fhi$ ($\L$ some countable language) such that the the quasi-order induced by elementary embeddability on $Mod_\fhi$ is complete analytic?

 Is it true that for every analytic quasi-order $R$ on a standard Borel space $X$ there is an $\L_{\omega_1 \omega}$-sentence $\fhi$ such that $R$ is Borel equivalent to elementary embeddability on $Mod_\fhi$?
\end{question}

Regarding Section
\ref{sectionanalysis}. There is an evident qualitative difference
between Corollary \ref{cordiscrete} on the one hand, and Corollaries
\ref{corultrametric}, \ref{corcompact} and \ref{corbanach} on the
other: in the first case we get a Borel class $\mathcal{C}$ which
is \emph{saturated} with respect to isometry, while in the second case
we get classes $\mathcal{C}$ which are very far from being saturated
with respect to the corresponding bijective morphisms (namely
isometries, homeomorphisms and linear isometries, respectively).

\begin{question}\label{questcor}
 Can one strenghten Corollaries \ref{corultrametric}, \ref{corcompact}
 and \ref{corbanach} by requiring that the Borel class $\mathcal{C}$
 is closed under, respectively, isometries, homeomorphisms and linear
 isometries? 
\end{question}

 What is missing to  positively answer this question is a
 method which allows one to saturate ``in a Borel way'' the class
 $\mathcal{C}$, something which is possible if the class of
 bijective morphisms would form a natural Polish group. For example,
 if the class of isometries between arbitrary ultrametric Polish
 spaces (viewed as elements of $F(\mathbb{U})$) would form a Polish
 group acting on $F(\mathbb{U})$, then it would be very likely that we
 could repeat the argument used in Theorem \ref{theor2} to prove that the
 saturation of $\mathcal{C}$ is still Borel. However, one should note
 that these isomorphisms are just \emph{partial} isomorphisms of
 $\mathbb{U}$ into itself, and in general they cannot be extended to
 an automorphism of $\mathbb{U}$, so the fact that the automorphisms of $\mathbb{U}$ form a Polish group acting on $F(\mathbb{U})$ is of no use here. Moreover, we can not see any other
 natural way to view these partial isomorphisms as a Polish group acting on $F(\mathbb{U})$, 
therefore we suspect that
 an aswer to Question \ref{questcor} would necessarily employ different
 techniques. 

We end with a question regarding Section \ref{sectionmonoid}. Call an action of a Polish
monoid \emph{functional} if it is the graph of a function (so that
functional actions coincide with the actions considered in
\cite{louros}). 

\begin{question}\label{quest3}
 Is it true that any analytic quasi-order $R$ on a standard Borel
 space $X$ is Borel equivalent to a quasi-order induced by a
 continuous (or just Borel) \emph{functional} action of a Polish
 monoid? 
\end{question}

The best result that we have in this direction is the following: 
the quasi-order of embeddability on countable graphs is induced
by the continuous functional action of a Polish monoid (although
neither the monoid nor the action are very ``natural''). This
strengthens a little bit \cite[Theorem 5.1]{louros}, since it gives a
more canonical and natural example of an  complete analytic 
quasi-order induced by a functional 
continuous action of a Polish monoid. Our example can also  be easily extended to cover other important cases, such as embeddability on combinatorial trees and so on, but it seems that it doesn't work in the general context of Borel invariant subsets of $Mod_\L$ (such a generality would clearly answer Question \ref{quest3} by Theorem \ref{theor2}).
Here is the construction (we
leave to the reader the 
proof of the fact that the proposed monoid is Polish and that the
action is continuous): the monoid $M \subseteq S^-_\infty \times
{}^\omega 2 \times X_G$ is defined by
\begin{align*}
 M = \{ (p,u,v) \mid & {\forall n \in \omega ({u(i) = 1} \iff {\exists m
\in \omega (p(m) = n))}} \wedge \\
& \wedge { \forall n,m \in \omega ({{u(n) =1} \wedge
{u(m) = 1}} \imp {v(\langle n,m \rangle) = 0})}\}.
\end{align*}

The real $u$ is used to determine whether or not an element is in the range
of the injection $p$, and is instrumental in proving that the
operation of the monoid defined below is continuous. The graph $v$
defines  the structure of the target graphs of the action of $M$
\emph{outside} the range of the embedding $p$ --- see the discussion
in Section \ref{sectionmonoid}. Define now first the functional action
of $M$ on $X_G$ as the function
\[ a \colon M \times X_G \to X_G \colon (g,x) \mapsto y, \]
where $g = (p,u,v)$ and $y$ is such that 
\[ y(\langle n,m \rangle) = 
\begin{cases}
x(\langle p^{-1}(n),p^{-1}(m) \rangle) & \text{if }u(n)= u(m) =
1,\\ 
v(\langle n,m \rangle) & \text{otherwise.}
\end{cases}
\]
 Finally, define the
product operation $(h,g) \mapsto hg$
in the unique way which 
really turns $a$ into an action, that is in such a way  that $a(h,a(g,x)) =
a(hg,x)$: if $g = (p_1,u_1,v_1)$ and
$h = (p_2,u_2,v_2)$ then $h  g = (q,t,w)$, where $q = p_2 \circ p_1$,
${t(n) = 1} \iff 
{\exists m(q(m) = n)}$, 
and 
\[ w(\langle n,m \rangle) =  
\begin{cases}
0 & \text{if } t(n) = t(m) = 1,\\
v_2(\langle n,m \rangle) & \text{if } v_2(n) =0 \text{ or } v_2(m) = 0,
\\
v_1(\langle p_2^{-1}(n), p_2^{-1}(m) \rangle) & 
\text{otherwise.}
\end{cases}
\]

\end{document}